\documentclass[a4paper,10pt]{amsart}
\usepackage{amsmath,amsthm,amsfonts}
\usepackage{amssymb}
\usepackage[utf8]{inputenc}
\usepackage{mathtools}

\newcommand{\ip}[2]{\langle #1,\, #2\rangle}	
\newcommand{\cm}{\bm{\nu}}

\usepackage{esint}
\usepackage{graphicx}
\usepackage{subfigure}
\usepackage{euscript}
\usepackage{units}

\usepackage[bookmarks=true,hidelinks]{hyperref}
\usepackage{cleveref}
\usepackage{bbm}
\usepackage{bm}
\usepackage{xcolor}
\usepackage{calrsfs}

\newtheorem{theorem}{Theorem}[section]
\newtheorem{proposition}[theorem]{Proposition}
\newtheorem{corollary}[theorem]{Corollary}
\newtheorem{lemma}[theorem]{Lemma}
\newtheorem*{lemma*}{Lemma}

\newtheorem{assumption}{Assumption}

\theoremstyle{definition}
\newtheorem{definition}[theorem]{Definition}
\newtheorem{notation}[theorem]{Notation}
\newtheorem{remark}[theorem]{Remark}

\numberwithin{equation}{section}

\def\Xint#1{\mathchoice
{\XXint\displaystyle\textstyle{#1}}
{\XXint\textstyle\scriptstyle{#1}}
{\XXint\scriptstyle\scriptscriptstyle{#1}}
{\XXint\scriptscriptstyle\scriptscriptstyle{#1}}
\!\int}
\def\XXint#1#2#3{{\setbox0=\hbox{$#1{#2#3}{\int}$ }
\vcenter{\hbox{$#2#3$ }}\kern-.56\wd0}}
\def\intavg{\Xint-}

\newcommand{\dom}{D}
\renewcommand{\geq}{\geqslant}
\renewcommand{\leq}{\leqslant}

\renewcommand{\epsilon}{\varepsilon}
\newcommand{\eps} {\varepsilon}
\renewcommand{\phi}{\varphi}
\newcommand{\R}{\mathbb{R}}
\newcommand{\N}{\mathbb{N}}

\DeclareMathOperator*{\esssup}{ess\, sup}

\newcommand{\ind}{\mathbbm{1}}

\newcommand{\torus}{\mathbb{T}}
\newcommand{\cyl}{\mathcal{T}_{\text{cyl}}}
\newcommand{\Grad}{\nabla}
\DeclareMathOperator{\Div}{div}

\newcommand{\hdiv}{L^2_{\Div}(\dom;U)}
\newcommand{\V}{H^1_{\Div}(\dom;U)}
\newcommand{\Vd}{H^1_{\Div}(\dom;U)}
\newcommand{\Lvec}{L^2(\dom;U)}
\newcommand{\Hvec}{H^1(\dom;U)}

\newcommand{\Hn}{L^2_{\Div}}

\newcommand{\Caratheodory}{\EuScript{H}}

\renewcommand{\geq}{\geqslant}
\renewcommand{\leq}{\leqslant}

\renewcommand{\phi}{\varphi}

\newcommand{\Meas}{\ensuremath{\mathcal{M}}}

\newcommand{\Ypair}[2]{\bigl\langle #1, #2\bigr\rangle}

\newcommand{\Borel}{\mathcal{B}}

\newcommand{\Prob}{\ensuremath{\mathcal{P}}}
\newcommand{\weakto}{\ensuremath{\rightharpoonup}}
\newcommand{\weaksto}{\ensuremath{\overset{*}{\weakto}}}

\usepackage{enumerate}

\newcommand{\vis}{\eps}

\newcommand{\norm}[1]{\left\| #1 \right\|}

 \newcommand{\Leray}{\mathbb{P}}
\newcommand{\bj}{\underline{j}}
\newcommand{\Corrmeas}{\EuScript{L}}

\begin{document}

\title[Vanishing viscosity limit of statistical solutions]{On the vanishing viscosity limit of statistical solutions of the incompressible Navier--Stokes equations}

\date{\today}
\author[U.S. Fjordholm]{Ulrik Skre Fjordholm} \address[Ulrik Skre Fjordholm]{\newline Department of Mathematics \newline
University of Oslo \newline Postboks 1053 Blindern, 0316 Oslo, Norway}
\email[]{ulriksf@math.uio.no}
\urladdr{https://www.mn.uio.no/math/english/people/aca/ulriksf/}
\author[S. Mishra]{Siddhartha Mishra} \address[Siddhartha Mishra]{\newline
Seminar for Applied Mathematics (SAM) \newline ETH Z\"urich
\newline HG G 57.2, R\"amistrasse 101, Z\"urich, Switzerland.}
\email[]{smishra@sam.math.ethz.ch}
\urladdr{http://www.sam.math.ethz.ch/~smishra}
\author[F. Weber]{Franziska Weber}
\address[Franziska Weber]{\newline Department of Mathematical Sciences \newline Carnegie Mellon University \newline 5000 Forbes Avenue, Pittsburgh, PA 15213, USA.}
\email[]{franzisw@andrew.cmu.edu}

\maketitle
\begin{abstract}
We study statistical solutions of the incompressible Navier--Stokes equation and their vanishing viscosity limit. We show that a formulation using correlation measures, which are probability measures accounting for spatial correlations, and moment equations is equivalent to statistical solutions in the Foia\c{s}--Prodi sense.	Under the assumption of weak scaling, a weaker version of Kolmogorov's self-similarity at small scales hypothesis that allows for intermittency corrections, we show that the limit is a statistical solution of the incompressible Euler equations. To pass to the limit, we derive a K\'arm\'an--Howarth--Monin relation for statistical solutions and combine it with the weak scaling assumption and a compactness theorem for correlation measures.
\end{abstract}
\section{Introduction}

The motion of an incompressible viscous fluid can be described by the Navier--Stokes equations
\begin{equation}\label{eq:ns}
\begin{split}
\partial_t u + \Div\bigl(u\otimes u\bigr) + \nabla p &= \vis\Delta u \\
\Div u &= 0 \\
u\big|_{t=0} &= u_0,
\end{split}
\end{equation}
where $u\colon \dom\to U\coloneqq\R^d$ is the fluid velocity and $p\colon \dom\to \R$, the pressure, acting as a Lagrange multiplier to enforce the divergence constraint $\Div u = 0$, and $u_0$ is the initial condition.
Here, we take the spatial domain $\dom$ to be the $d$-dimensional torus $\dom\coloneqq\mathbb{T}^d$, and we denote the phase space by $U\coloneqq\R^d$. The divergence is defined as $\Div u:=\sum_{i=1}^d \partial_{x_i} u^i$ and $\Grad p:=(\partial_{x_1} p,\dots,\partial_{x_d} p)^\top$ is the spatial gradient. The parameter $\vis>0$ denotes the \emph{viscosity} and is proportional to the reciprocal of the Reynolds number. It is well-known that many flows of interest are characterized by high to very-high Reynolds numbers. Hence, one is interested in studying what happens when viscosity $\vis$ equal to zero. In this formal limit $\vis\to 0$, one obtains the incompressible Euler equations, which are the prototypical models for an ideal fluid.

The question of whether the $\vis\to0$ limit is a good approximation of~\eqref{eq:ns} is of great practical relevance and has received considerable attention, both from a physical as well as mathematical point of view. Furthermore, it plays an essential role in computational fluid dynamics, as many numerical methods for the Euler equations, as well as large eddy simulations (LES) for Navier-Stokes equations, can be viewed as discretizations of~\eqref{eq:ns} with $\vis$ of the order of the discretization parameter. 

While in two spatial dimension, the convergence of a sequence of solutions $\{u^\vis\}_{\vis>0}$ of~\eqref{eq:ns} to a solution of the Euler equations has been proved rigorously for many settings, see e.g.,~\cite{Clopeau1998,Nussenzveig2020,Ciampa2021,Constantin2021}, it turns out to be very challenging in 3D.
Leray~\cite{Leray1934} proved in 1934 the existence of ``Leray--Hopf solutions'' of \eqref{eq:ns}, which are weak solutions of~\eqref{eq:ns}, i.e., they satisfy~\eqref{eq:ns} in the sense of distributions and in addition an energy inequality of the form
\begin{equation}
\label{eq:energydetns}
\int_{\dom} |u(t,x)|^2 dx +\vis \int_0^t\int_{\dom}|\Grad u(s,x)|^2 dxds\leq \int_{\dom} |u_0(x)|^2 dx,\quad t\geq 0.
\end{equation}
(Here and in the remainder, we suppress the dependence of $u^\vis$ on $\vis$ for convenience and write just $u$.)
Hence, if the initial data $u_0$ lies in $\hdiv$, such solutions satisfy $u\in L^\infty\big([0,T];L^2_{\Div}(\dom;\R^d)\big)\cap L^2\big([0,T];H^1_{\Div}(\dom;\R^d)\big)$. (Here, $L^2_{\Div}$ and $H^1_{\Div}$ are the weakly divergence free functions in $L^2$ and $H^1$, respectively.) However, as $\vis\to 0$, the $L^2$-bound on the gradient of $u$, that stems from the energy inequality~\eqref{eq:energydetns}, no longer suffices for deriving sufficient compactness of the sequence $\{u\}_{\vis>0}$ in $L^2$---which would be needed to pass to the limit in the nonlinear terms. It appears that, at least as far as global solutions are concerned, there is currently no means of gaining sufficient compactness through other conserved quantities or bootstrapping; in fact, it is unclear if global solutions of higher regularity than the one given by~\eqref{eq:energydetns} exist in 3D. Closely related to this issue is the lack of stability estimates, i.e., well-posedness of Leray-Hopf solutions~\cite{Fefferman2006,Ladyzhenskaya2003}. The main obstruction to better regularity or stability estimates is caused by the nonlinear convective term $(u\cdot \Grad)u$. The role of the nonlinear term and possible instabilities in the Leray-Hopf solutions are often related to the issue of \emph{turbulence} in fluid flows. 

The theory of mathematical turbulence was initiated in the 1930s and 1940s by Taylor, Richardson, Kolmogorov and others, see~\cite{Frisch} and references therein, and has since influenced fluid mechanics, as well as atmospheric sciences and plasma physics heavily. In his sequence of three papers~\cite{K41a,K41b,K41c}, nowadays referred to as \textit{K41}, Kolmogorov took a probabilistic approach to turbulence and formulated basic hypotheses about fluid flow at high Reynolds numbers and derived predictions based on these. Many of these have later been confirmed by experiments. The idea of studying equations~\eqref{eq:ns} in a probabilistic setting
has since been taken up again in many works, in different frameworks, by adding stochastic forcing terms to~\eqref{eq:ns}, see e.g.~\cite{Flandoli1995,Motyl2012,Brzezniak2013}, or taking uncertain or measure-valued initial data, e.g.~\cite{DiPerna1987}. In the latter case, the solution of~\eqref{eq:ns} may not be a function any more but instead a time-parametrized probability measure on the phase space. Global existence of such measure-valued solutions for incompressible flows has been shown in 3D, and even the passage to the limit $\vis\to 0$ can be made rigorous in this case~\cite{DiPerna1987}. However, measure-valued solutions are generally not unique, which can be shown by counterexample even in the case of Burgers' equation~\cite{Fjordholm2017}. Hence, measure-valued solutions are too broad a solution concept to resolve the problem of non-uniqueness, and more information or constraints need to be added.

To overcome this, in~\cite{FLM17}, it was suggested to take into account the (time) evolution of all possible multi-point spatial correlations. Instead of a single probability measure on the phase space $U$, such a \emph{statistical solution} is a family of probability measures on the phase space and products of the phase space $U^k$, for $k\in\N$, corresponding to the multi-point correlations. Hence, one can interpret the solution as a measure-valued solution augmented with information about higher order spatial correlations.
From a practical point of view, this approach is very natural, as often only averaged quantities of interest of the fluid flow can be observed. Moreover, it is also in line with Kolmogorov's turbulence theory, as this theory studies statistical properties of the fluid and makes predictions about these.
The system of equations that arises for the higher order correlations is also known as the  \emph{Friedman--Keller infinite chain of moment equations}~\cite{Keller1924,Vishik1988} and finite closure relation for this infinite family of equations have been studied for small and large Reynolds numbers in~\cite{Fursikov1990,Fursikov1991,Fursikov1987,Fursikov1994}.

An alternative point of view in this context, is to consider instead probability measures on a space of suitable initial conditions; in the case of~\eqref{eq:ns} this would be $\hdiv$. Equation~\eqref{eq:ns} is then interpreted as a Liouville equation on an infinite dimensional function space and the solution is a mapping assigning to each time $t$ a probability measure on $\hdiv$. This setting was first considered by Prodi~\cite{Prodi1961} and later on extensively studied by Foia\c{s} and collaborators~\cite{Foias1972,Foias1973,FoiasProdi1976,FoiasTemam2001}, see also~\cite{Hopf1952}. A closely related notion of statistical solutions was studied by Vishik and Fursikov~\cite{Vishik1988}. Foia\c{s} and his collaborators proved existence of such solutions in 2D and 3D, uniqueness in 2D, and further properties related to turbulence~\cite{FoiasTemam2001}. The relations between the Foia\c{s} and Prodi notion of statistical solutions and the Vishik-Fursikov version were explored in~\cite{Foias2013,Bronzi2014,Bronzi2016}. The latter work etends the notion of statistical solutions to other relevant PDEs in fluid mechanics. 

Given this plurality of definitions of statistical solutions, it is natural to examine, if and under what conditions, these solution concepts are equivalent. \emph{The first goal of this paper is to prove that both these concepts of statistical solutions of the incompressible Navier-Stokes equations \eqref{eq:ns} are equivalent} as long as a statistical version of the of the energy inequality~\eqref{eq:energydetns} holds. 

\emph{The second and main goal of this paper is to investigate the vanishing viscosity limit of the statistical solutions of incompressible Navier-Stokes equations}. Under an \emph{weak scaling} assumption on the Navier-Stokes statistical solutions, we will use compactness criteria, presented recently in \cite{systemspaper}, to prove that vanishing viscosity limits of the statistical solutions of Navier-Stokes equations are statistical solutions of the incompressible Euler equations. 

Our weak scaling assumption is a significantly weaker version of the scaling hypothesis of Kolmogorov's 1941 theory and allows for \emph{intermittent} corrections. Our main technical tool is a statisical version of the well-known K\'arm\'an--Howarth--Monin relation~\cite{Frisch,deKarman1938,MoninYaglom}, that relates the evolution of 2-point correlations to the longitudinal structure function $S_{\|}^3$, which is, roughly speaking, defined as
\begin{equation*}
S_{\|}^3(|\ell|)\coloneqq\Bigl\langle \Bigl((u(x+\ell)-u(x))\cdot \widehat{\ell}\Bigr)^3\Bigr\rangle,\quad \widehat{\ell}\coloneqq\frac{\ell}{|\ell|}.
\end{equation*}
Here $\langle\cdot\rangle$ is a suitable average of the flow. 

Thus, by characterizing this vanishing viscosity limit, we establish a rigorous relationship between the incompressible Navier-Stokes and Euler equations, while accommodating physically observed facts about turbulent flows in this description.

The remainder of this article is organized as follows: In Section~\ref{sec:prelim}, we introduce the concept of correlation measures and in Section~\ref{sec:statsol} we show the equivalence of statistical solutions as introduced by Foia\c{s} and Prodi with families of correlation measures satisfying the Friedman--Keller chain of moment equations. Then in Section~\ref{sec:VV}, we consider the passage to the limit $\vis\to 0$ and conclude with an appendix with technical results.

\section{Correlation measures}\label{sec:prelim}

In this section, we recall the definition of correlation measures and some important properties of them from~\cite{FLM17,systemspaper}.
We start by introducing the necessary notation.
\subsection{Notation}
For $k\in\N$, $k\geq 1$, we denote the tensor products
\begin{equation*}
D^k=\underbrace{D\times\dots\times D}_{k\, \textrm{times}},\quad 
U^k=\underbrace{U\otimes\dots\otimes U}_{k\, \textrm{times}}.
\end{equation*}
If $X$ is a topological space then we let $\Borel(X)$ denote the Borel $\sigma$-algebra on $X$, we let $\Meas(X)$ denote the set of signed Radon measures on $(X,\Borel(X))$, and we let $\Prob(X)\subset\Meas(X)$ denote the set of all probability measures on $(X, \Borel(X))$, i.e.~all $0\leq\mu\in\Meas(X)$ with $\mu(X)=1$ (see e.g.~\cite{ambrosio_gradient_flows,Bil08,Kle}). For $k\in\N$ and a multiindex $\alpha\in \{0,1\}^k$ we write $|\alpha|=\alpha_1+\dots+\alpha_k$ and $\bar{\alpha} = \mathbbm{1}-\alpha=(1-\alpha_1,\dots,1-\alpha_k)$, and we let $x_{\alpha}$ be the vector of length $|\alpha|$ consisting of the elements $x_i$ of $x$ for which $\alpha_i$ is non-zero. For a vector $x=(x_1,\dots,x_k)$ we write $\widehat{x}_i = (x_1,\dots,x_{i-1},x_{i+1},\dots,x_k)$. For a vector $\xi=(\xi_1,\dots,\xi_k)$ we write $|\xi^\alpha| = |\xi_1|^{\alpha_1}\cdots|\xi_k|^{\alpha_k}$ with the convention $0^0=1$.

\subsubsection{Carath\'eodory functions}
If $E$ and $V$ are Euclidean spaces then a measurable function $g\colon E\times V\to\R$ is called a \emph{Carath\'eodory function} if $\xi\mapsto g(x,\xi)$ is continuous for a.e.~$y\in E$ and $y\mapsto g(y,\xi)$ is measurable for every $\xi\in V$ (see e.g.~\cite[Section 4.10]{AliBor}). Given $k\in\N$ and a Carath\'eodory function $g=g(x,\xi)\colon D^k\times U^k\to\R$ we define the functional $L_{g} \colon  L^p(D;U) \to \R$ by
\begin{equation}\label{eq:Lgdef}
L_{g}(u) \coloneqq \int_{D^k} g(x_1,\dots,x_k,u(x_1),\dots,u(x_k))\,dx.
\end{equation}
(It is not obvious that $L_g$ is continuous, or even well-defined; see \cite{FLM17}.) We denote the set of Carath\'eodory functions depending on space and time by $\Caratheodory^k_0([0,T),D;U) \coloneqq L^1([0,T)\times D^k;C_0(U^k))$ and its dual space by $\Caratheodory^{k*}_0([0,T),D;U) \coloneqq L^\infty_w([0,T)\times D^k;\Meas(U^k))$ (see e.g.~\cite{Bal89}).

In the following, we will focus on a specific type of Carath\'eodory functions. In particular, for $p\geq1$ we let $\Caratheodory^{k,p}([0,T],D;U)$ denote the space of Carath\'eodory functions $g\colon [0,T]\times D^k\times U^k \to \R$ satisfying
\begin{equation}\label{eq:gboundedt}
|g(t,x,\xi)| \leq \sum_{\alpha\in\{0,1\}^k} \phi_{|\bar{\alpha}|}(t,x_{\bar{\alpha}})|\xi^\alpha|^p \qquad \forall\ x\in D^k,\ \xi\in U^k
\end{equation}
for nonnegative functions $\phi_i\in L^\infty([0,T]; L^1(D^i))$, $i=0,1,\dots,k$. We let $\Caratheodory^{k,p}_1([0,T],D;U)\subset \Caratheodory^{k,p}([0,T],D;U)$ denote the subspace of functions $g$ satisfying the local Lipschitz condition
\begin{equation}\label{eq:glipschitzt}
\begin{split}
\big|g(t,x,\zeta)-g(t,y,\xi)\big| \leq &\ \psi(t)\sum_{i=1}^k|\zeta_i-\xi_i|\max\big(|\xi_i|,|\zeta_i|\big)^{p-1} h(t,\widehat{x}_i,\widehat{\xi}_i) \\
&+ O(|x-y|)\widetilde{h}(t,x,\xi)
\end{split}
\end{equation}
for every $x\in D^k$, $y\in B_r(x)$ for some $r>0$, for some nonnegative $h \in \Caratheodory^{k-1,p}([0,T],D;U)$ and $0\leq \psi(t)\in L^{\infty}([0,T])$ and some $\widetilde{h}\in\Caratheodory^{k,p}([0,T],D;U)$. (Note that the term $\widetilde{h}$ was not present in \cite[Definition 2.2]{systemspaper}, but one can generalize the results of that paper to include such a term.)

We also denote for a parametrized probability measure $\nu^k\in L^\infty_w([0,T)\times D^k;\Meas(U^k))$ and a Carath\'eodory function $g$ the pairing
\[
\Ypair{\nu^k}{g}_{\Caratheodory^k} = \int_{D^k} \Ypair{\nu^k_{t,x}}{g(t,x)}\,dx
\]
(where $\Ypair{\nu^k_x}{g(t,x)} = \int_{U^k} g(t,x,\xi)\,d\nu^k_{t,x}(\xi)$  is the usual duality pairing between Radon measures $\Meas(U^k)$ and continuous functions $C_0(U^k)$).
\subsection{Definitions}

We are now in a position to define time-dependent correlation measures.

\begin{definition}\label{def:timdeptrunccorrmeas}
	A \emph{time-dependent correlation measure} is a collection $\cm=(\nu^1,\nu^2,\dots)$ of functions $\nu^k\in \Caratheodory^{k*}_0([0,T),D;U)$ such that
	\begin{enumerate}[{\rm (i)}]  

		\item $\nu^k_{t,x} \in \Prob(U^k)$ for a.e.\ $(t,x)\in [0,T]\times D^k$, and the map $x \mapsto \Ypair{\nu^k_{t,x}}{f}$ is measurable for every $f\in C_b(U^k)$ and almost every $t\in [0,T]$. (In other words, $\nu^k_t$ is a Young measure from $D^k$ to $U^k$.)
		\item \textit{$L^p$ integrability:}
		\begin{equation}\label{eq:corrlpboundtimedep}
		\esssup_{t\in[0,T)}\left(\int_{D} \Ypair{\nu^1_{t,x}}{|\xi|^p}\,dx\right)^{1/p} \leq c< +\infty
		\end{equation}
		\item\textit{Diagonal continuity (DC):}
		\begin{equation}\label{eq:timedepdc}
		\int_0^{T'}\omega_r^p\big(\nu^2_t\big)\,dt \to 0 \qquad \text{ as } r\to0 \text{ for all } T'\in(0,T),
		\end{equation}
		where
		\[
		\omega_r^p(\nu^2_t) \coloneqq \int_D \intavg_{B_r(x)} \Ypair{\nu^2_{t,x,y}}{|\xi_1-\xi_2|^p}\,dy\,dx
		\]
		is called the \emph{modulus of continuity} of $\cm$.
	\end{enumerate}
	We denote the set of all time-dependent correlation measures by $\Corrmeas^{p}([0,T),D;U)$.
\end{definition}

In~\cite{systemspaper} (and see~\cite{FLM17} for a time-independent version), the following equivalence between time-dependent correlation measures and parametrized probability measures on $L^p(D)$ was proved:

\begin{theorem}\label{thm:equivalencemeasures}
	For every time-dependent correlation measure $\cm\in\Corrmeas^p([0,T),D;U)$ there is a unique {\rm(}up to subsets of $[0,T)$ of Lebesgue measure $0${\rm)} map $\mu\colon [0,T)\to\Prob(L^p(D;U))$ such that
	\begin{enumerate}[{\rm (i)}] 
		\item the map
		\begin{equation}\label{eq:mumeasurable}
		t\mapsto\Ypair{\mu_t}{L_g} = \int_{L^p} \int_{D^k}g(x,u(x))\,dx\,d\mu_t(u)
		\end{equation}
		is measurable for all $g\in\Caratheodory^k_0(D;U)$,
		\item $\mu$ is $L^p$-bounded:
		\begin{equation}\label{eq:mulpbound}
		\esssup_{t\in[0,T)}\int_{L^p}\|u\|_{L^p}^p\,d\mu_t(u) \leq c^p < \infty
		\end{equation}
		\item $\mu$ is dual to $\cm$: the identity
		\begin{equation}\label{eq:corrmeastimeduality}
		\int_{D^k}\Ypair{\nu_t^k}{g(x)}\,dx = \int_{L^p} \int_{D^k}g(x,u(x))\,dx\,d\mu_t(u)
		\end{equation}
		holds for a.e.~$t\in[0,T)$, every $g\in\Caratheodory^k_0(D;U)$ and all $k\in\N$.
	\end{enumerate}
	Conversely, for every $\mu\colon [0,T)\to\Prob(L^p(D;U))$
	satisfying (i) and (ii), there is a unique correlation measure $\cm\in\Corrmeas^p([0,T),D;U)$ satisfying (iii).
\end{theorem}

We also have the following ``Compactness'' Theorem for time-dependent correlation measures~\cite[Theorem 2.21]{systemspaper}
\begin{theorem}\label{thm:timedepcmcompactness}
	Let $\cm_n \in \Corrmeas^{p}([0,T),D;U)$ for $n=1,2,\dots$ be a sequence of correlation measures such that
	\begin{align}
	\sup_{n\in\N}\esssup_{t\in[0,T)}\left(\int_{D} \Ypair{\nu^1_{n;t,x}}{|\xi|^p}\,dx\right)^{1/p} \leq c< +\infty	\label{eq:timedepuniformlpbound} \\
	\lim_{r\to0}\limsup_{n\to\infty}\int_0^{T'}\omega_r^p\bigl(\nu^2_{n,t}\bigr)\,dt = 0 \label{eq:timedepuniformdc}
	\end{align}
	for some $c>0$ and all $T'\in[0,T)$. Then there exists a subsequence $(n_j)_{j=1}^\infty$ and some $\cm\in\Corrmeas^{p}([0,T),D;U)$ such that
	\begin{enumerate}[(i)]   
		\item\label{prop:tdwsconv} $\cm_{n_j} \weaksto \cm$ as $j\to\infty$, that is, $\Ypair{\nu^k_{n_j}}{g}_{\Caratheodory^k} \to \Ypair{\nu^k}{g}_{\Caratheodory^k}$ for every $g\in\Caratheodory^k_0([0,T),D;U)$ and every $k\in\N$
		\item\label{prop:tdlpbound} $\Ypair{\nu^1_t}{|\xi|^p}_{\Caratheodory^1} \leq c^p$ for a.e.~$t\in[0,T)$
		\item\label{prop:tddc} $\int_0^{T'}\omega_r^p\big(\nu^2_t\big)\,dt \leq \liminf_{n\to\infty} \int_0^{T'}\omega_r^p\big(\nu^2_{n,t}\big)\,dt$ for every $r>0$ and $T'\in[0,T)$
		\item\label{prop:tdliminf} for $k\in\N$, let $\phi\in L^1_{\textrm{loc}}([0,T)\times D^k)$ and $\kappa\in C(U^k)$ be nonnegative, and let $g(t,x,\xi) \coloneqq \phi(t,x)\kappa(\xi)$. Then
		\begin{equation}\label{eq:tdnuklimitbound}
		\Ypair{\nu^k}{g}_{\Caratheodory^k} \leq \liminf_{j\to\infty} \Ypair{\nu_{n_j}^k}{g}_{\Caratheodory^k}.
		\end{equation}
		\item \label{prop:strongconv}
		Assume moreover that $D\subset\R^d$ is compact, $T<\infty$ and that $\cm_n$ have uniformly bounded support, in the sense that
		\begin{equation}
		\|u\|_{L^p} \leq R \qquad \text{for $\mu^n_t$-a.e.\ } u\in L^p(D;U) \text{ for every } n\in\N, ~ a.e~t \in (0,T),  \label{eq:tLpLinftybound}
		\end{equation}
		with $\mu_t^n\in \Prob_T(L^p(D;U))$ being dual to $\cm_n$,
		then the following \emph{observables converge strongly:}
		\begin{equation}\label{eq:tnuklimit}
		\lim_{j\to\infty}\int_{D^k} \left|\int_0^T\left(\Ypair{\nu^k_{n_j;t,x}}{g(t,x)} -  \Ypair{\nu^k_{t,x}}{g(t,x)}\right)\,dt\right| \,dx = 0
		\end{equation}
		for every $g \in \Caratheodory^{k,p}_1([0,T],D;U)$.
	\end{enumerate}
\end{theorem}

\section{Statistical solutions}\label{sec:statsol}
The goal of this section is to show that the statistical solutions of Navier--Stokes as introduced by Foia\c{s} and Prodi~\cite{Foias1972,Foias1973,FoiasTemam2001,FoiasProdi1976,Prodi1961} are equivalent to families of correlation measures as introduced in~\cite{FLM17} that satisfy the Friedman--Keller system of moment equations. For the sake of simplicity we will assume that the support of the initial measure $\mu_0$ lies in a bounded set $B\subset \Lvec$, that is,
\begin{equation}\label{eq:initboundedsupp}
\text{supp}(\mu_0)\subset B\subset \Lvec.
\end{equation}

\subsection{The Leray projector}\label{sec:lerayproj}
We recall first that the \emph{Helmholtz--Leray projector}, or simply \emph{Leray projector}, is the linear map $\Leray\colon \Lvec\to \hdiv\coloneqq\{v\in \Lvec : \Div v = 0,\,\int v dx =0\}$  that projects a vector field $f\in \Lvec$ to its divergence free component, that is $f=\Leray f +\Grad \psi_f$ with $\Div (\Leray f)=0$ and $\Leray f, \Grad\psi_f\in \Lvec$. One can show that $\Grad\psi_f$ is orthogonal in $\Lvec$ to any function $u\in \hdiv$,
\begin{equation*}
\int_{\dom}u \cdot\Grad\psi_f\,dx = 0.
\end{equation*}
For functions in the tensor product space $L^2(D^k;U^k)$ we let $\Leray_{x_i}$ denote the Leray projector in the $i$th component, i.e., $\varphi = \Leray_{x_i}\varphi+\Grad_{x_i}\psi_{\varphi,i}$ where $\Div_{x_i}(\Leray_{x_i}\varphi) = 0$, and $\psi_{\varphi,i}\in L^2(D^k; U^{k-1})$ with $\Grad_{x_i}\psi_{\varphi,i}\in L^2(D^k;U^k)$.

\subsection{Definitions}
We will start by recalling the different definitions of statistical solutions introduced in \cite{FoiasTemam2001,Foias2013,FLM17}.

\begin{definition}[Definition 3.2 in \cite{FLM17}]\label{def:fksys} Let $\vis \geq 0$. The \emph{Friedman--Keller system of moment equations}, defined for time-dependent correlation measures $\cm\in\Corrmeas^2([0,T),D;U)$, is the hierarchy of equations
	\begin{equation}\label{eq:FKsys}
	\begin{split}
	&\int_0^T\!\!\int_{\dom^k}\!\!\int_{U^k}\!\!(\xi_1\otimes \cdots\otimes \xi_k):\frac{\partial \varphi}{\partial t}(t,x)\,d\nu_{t,x}^k(\xi) \,dx \,dt \\
	&+\int_{\dom^k}\!\!\int_{U^k}\!\!(\xi_1\otimes \cdots\otimes \xi_k):\varphi(0,x)\,d\nu_{0,x}^k(\xi) \,dx \\
	&+\sum_{i=1}^k\int_0^T\!\!\int_{\dom^k}\!\!\int_{U^k}(\xi_1\otimes\cdots\otimes(\xi_i\otimes\xi_i)\otimes\cdots\xi_k):\Grad_{x_i}\varphi(t,x)\,d\nu_{t,x}^k(\xi) \,dx \,dt\\
	=&\ -\vis\sum_{i=1}^k\int_0^T\!\!\int_{\dom^k}\!\!\int_{U^k}\!\!(\xi_1\otimes \cdots\otimes \xi_k):\Delta_{x_i}\varphi(t,x)\,d\nu_{t,x}^k(\xi) \,dx \,dt
	\end{split}
	\end{equation}
	for all $k\in\N$, for all $\varphi\in C^2_c([0,T)\times \dom^k;U^k)$ with $\Div_{x_i}\varphi=0$ for all $i=1,\dots, k$,
	along with the divergence constraint
	\begin{equation}
	\label{eq:divfreeconst}
	\int_{\dom^{k}}\int_{U^{k}} \xi_1\otimes\dots\otimes\xi_{\ell}\otimes\alpha_{\ell+1}(\xi_{\ell+1})\otimes\dots \otimes \alpha_k(\xi_k) \,d\nu_{t,x}^k(\xi)\cdot\Grad_{x_1,\dots, x_\ell}\varphi(x) \,dx = 0,
	\end{equation}
	where $\Grad_{x_1,\dots,x_\ell}=(\Grad_{x_{1}},\dots,\Grad_{x_\ell})^\top$, $1\leq \ell\leq k\in\N$, for all $\varphi\in H^1(\dom^k;U^{k-\ell})$, $\alpha_j\in C(U;U)$, with $\alpha_j(v)\leq C(1+|v|^2)$ for all $j=1,\dots, k$. 

	If $\cm$ solves the Friedman--Keller system of moments equations and in addition satisfies the energy inequality
	\begin{equation}\label{eq:energycm}
	\begin{aligned}
	&\sum_{k=0}^K a_k\int_{\dom^k}\int_{U^k}|\xi_1|^2\dots|\xi_k|^2 \,d\nu_{t,x}^k(\xi) \,dx\\
	&+ 2\vis \sum_{k=0}^K a_k\sum_{i=1}^k\sum_{j=1}^d \lim_{h\rightarrow 0}\frac{1}{h^2}\int_0^t\!\!\int_{\dom^{k}}\!\int_{U^{k+1}} \!\!\!|\xi_1|^2\dots |\xi_i - \xi_{k+1}|^2\dots |\xi_k|^2\,d\nu_{t,(x,x_i+h\mathbf{e}_j)}^{k+1}(\xi,\xi_{k+1})\,dx \,ds\\
	&\leq 
	\sum_{k=0}^K a_k\int_{\dom^k}\!\!\int_{U^{k}}|\xi_1|^2\dots|\xi_k|^2 \,d\nu_{0,x}^k(\xi)\,dx
	\end{aligned}
	\end{equation}
	for all $K\in\N$ and $a_k\in\R$, $k=0,\dots, K$ such that $p_K(s)=\sum_{k=0}^K a_k s^k$ is a nonnegative, nondecreasing polynomial for $s\in [0,R]$ for $R$ sufficiently large related to the support of the correlation measure (see~\eqref{eq:tLpLinftybound}),
	then we call $\cm$ a \emph{Friedman--Keller statistical solution} of the Navier--Stokes (when $\eps>0$) or Euler (when $\eps=0$) equations.
\end{definition}
\begin{remark}
	{Friedman--Keller statistical solutions} are analogous to the definition of statistical solutions for hyperbolic systems of conservation laws as introduced in~\cite{FLM17,systemspaper} for compressible flows.
\end{remark}
\begin{remark}\label{rem:weakcont}
By a standard argument for weak solutions to continuity equations, the map $t \mapsto \ip{\nu^k_{t,\cdot}}{\xi_1\otimes\cdots\otimes\xi_k}$ is weakly continuous for every $k\in\N$; see e.g.~\cite[Remark 2.2]{AmbrosioCrippa2008}.
\end{remark}
\begin{remark}[Formulation of~\eqref{eq:FKsys} with non-divergence free test functions]\label{rem:lerayprojection}
Denoting $\Leray_k\coloneqq\Leray_{x_1}\dots\Leray_{x_k}$ (cf.~Section \ref{sec:lerayproj}), we can replace the divergence-free test function $\varphi\in C^2_c([0,T]\times \dom^k;U^k)$ in~\eqref{eq:FKsys} by $\Leray_k\varphi$ for an arbitrary $\varphi\in C^2_c([0,T]\times \dom^k;U^k)$. Using the fact that the Leray projection is self-adjoint and that, by \eqref{eq:divfreeconst},
	\begin{equation*}
	\Div_{x_i}\int_{U^k}\!\!(\xi_1\otimes \cdots\otimes \xi_k) \,d\nu_{t,x}^k(\xi) = 0 \qquad \forall\  i=1,\dots, k,\quad \text{a.e. }\, (t,x)\in [0,T]\times\dom^k
	\end{equation*}
	we observe that
	\begin{equation*}
	\Leray_k \int_{U^k}\!\!(\xi_1\otimes \cdots\otimes \xi_k) \,d\nu_{t,x}^k(\xi) = \int_{U^k}\!\!(\xi_1\otimes \cdots\otimes \xi_k) \,d\nu_{t,x}^k(\xi), \quad \text{a.e. }\, t\in [0,T],\,x\in \dom^k.
	\end{equation*}
	Therefore,
	\begin{align*}
	&\int_0^T\!\!\int_{\dom^k}\!\!\int_{U^k}\!\!(\xi_1\otimes \cdots\otimes \xi_k):\frac{\partial \Leray_k\varphi}{\partial t}(t,x)\,d\nu_{t,x}^k(\xi) \,dx \,dt\\
	= 	&\int_0^T\!\!\int_{\dom^k}\!\!\Leray_k\left(\int_{U^k}\!\!(\xi_1\otimes \cdots\otimes \xi_k)\,d\nu_{t,x}^k(\xi)\right):\frac{\partial \varphi}{\partial t}(t,x) \,dx \,dt \\
	= &\int_0^T\!\!\int_{\dom^k}\!\!\int_{U^k}\!\!(\xi_1\otimes \cdots\otimes \xi_k)\,d\nu_{t,x}^k(\xi):\frac{\partial \varphi}{\partial t}(t,x) \,dx \,dt.
	\end{align*}
	Similarly,
	\begin{equation*}
	\int_{\dom^k}\!\!\int_{U^k}\!\!(\xi_1\otimes \cdots\otimes \xi_k):\Leray_k\varphi(0,x)\,d\nu_{0,x}^k(\xi) \,dx = \int_{\dom^k}\!\!\int_{U^k}\!\!(\xi_1\otimes \cdots\otimes \xi_k):\varphi(0,x)\,d\nu_{0,x}^k(\xi) \,dx,
	\end{equation*}
	and
	\begin{align*}
	&\sum_{i=1}^k\int_0^T\!\!\int_{\dom^k}\!\!\int_{U^k}\!\!(\xi_1\otimes \cdots\otimes \xi_k):\Delta_{x_i}\Leray_k\varphi(t,x)\,d\nu_{t,x}^k(\xi) \,dx \,dt\\
	= &\sum_{i=1}^k\int_0^T\!\!\int_{\dom^k}\!\!\int_{U^k}\!\!(\xi_1\otimes \cdots\otimes \xi_k):\Delta_{x_i}\varphi(t,x)\,d\nu_{t,x}^k(\xi) \,dx \,dt,
	\end{align*}
	the last one being true due to the fact that the Laplacian and the Leray projection commute on the torus.
	Hence, the weak formulation~\eqref{eq:FKsys} can be rewritten as
	\begin{equation}
	\label{eq:FKsyspressure}
	\begin{split}
	&\int_0^T\!\!\int_{\dom^k}\!\!\int_{U^k}\!\!(\xi_1\otimes \cdots\otimes \xi_k):\frac{\partial \varphi}{\partial t}(t,x)\,d\nu_{t,x}^k(\xi) \,dx \,dt+\int_{\dom^k}\!\!\int_{U^k}\!\!(\xi_1\otimes \cdots\otimes \xi_k):\varphi(0,x)\,d\nu_{0,x}^k(\xi) \,dx \\
	&+\sum_{i=1}^k\int_0^T\!\!\int_{\dom^k}\!\!\int_{U^k}(\xi_1\otimes\cdots\otimes(\xi_i\otimes\xi_i)\otimes\cdots\xi_k):\Grad_{x_i}\varphi(t,x)\,d\nu_{t,x}^k(\xi) \,dx \,dt\\
	=&-\vis\sum_{i=1}^k\int_0^T\!\!\int_{\dom^k}\!\!\int_{U^k}\!\!(\xi_1\otimes \cdots\otimes \xi_k):\Delta_{x_i}\varphi(t,x)\,d\nu_{t,x}^k(\xi) \,dx \,dt\\
	&+\sum_{i=1}^k\int_0^T\!\!\int_{\dom^k}\!\!\int_{U^k}(\xi_1\otimes\cdots\otimes(\xi_i\otimes\xi_i)\otimes\cdots\xi_k):\Grad_{x_i}\Grad_{x_i}\psi_{\varphi,i}(t,x)\,d\nu_{t,x}^k(\xi) \,dx \,dt
	\end{split}
	\end{equation}
	where $\varphi\in C^2_c([0,T]\times\dom^k;U^k)$. The terms
	\begin{equation}
	\sum_{i=1}^k\int_0^T\!\!\int_{\dom^k}\!\!\int_{U^k}(\xi_1\otimes\cdots\otimes(\xi_i\otimes\xi_i)\otimes\cdots\xi_k):\Grad_{x_i}\Grad_{x_i}\psi_{\varphi,i}(t,x) \,d\nu_{t,x}^k(\xi) \,dx \,dt
	\end{equation}
	correspond to the pressure in the deterministic setting.

\end{remark}

To define statistical solutions in the sense of Foia\c{s} and Prodi, we need to introduce some notation.
We denote by $\hdiv$ the space of divergence free $L^2(\dom;U)$-vector fields and by $\V$ the space of divergence free functions in $\Hvec$ (these can be obtained as the closures of $C^\infty(\dom;U)\cap \{\Div u=0\}$ in $\Lvec$ and $\Hvec$, respectively, with suitable integral conditions:
\begin{equation*}
\begin{split}
\hdiv& = \left\{u\in L^2(\torus^d)\ :\  \Div u = 0,\  \int_{\torus^d} u(x)\,dx=0\right\},\\
\V &=\left\{u\in H^1(\torus^d)\ :\  \Div u = 0,\  \int_{\torus^d} u(x)\,dx=0\right\},
\end{split}
\end{equation*}
for periodic boundary conditions.
We denote the $L^2$-inner product by
\begin{equation*}
(u,v) = \int_{\dom} u(x) v(x) \,dx,
\end{equation*}
and for $\vis>0$ the $H^1$-inner product by
\begin{equation*}
a(u,v)= \vis\sum_{i=1}^d\int_{\dom}\frac{\partial u}{\partial x^i}\cdot\frac{\partial u}{\partial x^i} \,dx,
\end{equation*}
Define the Stokes operator $A$ by
\begin{equation*}
\begin{split}
Au &= -\Leray \Delta u, \quad \text{for all }u\in D(A)= \V\cap H^2(\dom;U),\\
\vis(Au,v)&= a(u,v),\quad \text{for all } u,v\in D(A^{1/2}),
\end{split}
\end{equation*}
where $\Leray$ is the Leray projector, and the skew-symmetric trilinear form $b$ by
\begin{equation}\label{eq:defb}
\begin{split}
b(u,v,w) &\coloneqq \int_{\dom} (u\cdot \Grad) v\cdot w \,dx = (B(u,v),w), \quad u,v,w\in D(A^{1/2})\\
B(u)&\coloneqq B(u,u).
\end{split}
\end{equation}
We can then write the Navier--Stokes equations in the functional formulation: Let $T>0$, $u_0\in \hdiv$, find $u\in L^\infty([0,T];\hdiv)\cap L^2([0,T];\V)$ with $u'\coloneqq\frac{d}{dt}u\in L^1([0,T];D(A^{-1/2}))$ such that
\begin{equation}
\label{eq:NSEoperatorform}
u'+ \vis Au+B(u)= 0
\end{equation}
and $u(0)= u_0$ in a suitable sense. This corresponds to the weak formulation
\begin{equation}
\label{eq:weakform}
\frac{d}{dt}(u,v) +a(u,v)+b(u,u,v) =0 \qquad \text{for all } v\in \V.
\end{equation}
If we denote
\begin{equation}
\label{eq:F}
F(t,u)\coloneqq-\vis A u -B(u),
\end{equation}
the functional formulation becomes
\begin{equation}
\label{eq:NSEfunctional}
u'(t)=F(t,u(t)).
\end{equation}
We need the following class of test functions:

\begin{notation}\cite{FoiasTemam2001}
	Let $\cyl$ denote	
	the class of cylindrical test functions consisting of the real-valued functionals $\Phi=\Phi(u)$ that depend on a finite number $k\in \N$ of components of $u$, that is,
	\begin{equation*}
	\Phi(u)= \phi\bigl((u,g_1),\dots, (u,g_k)\bigr),
	\end{equation*}
	where $\phi\in C^1_c(\R^k)$ and $g_1,\dots, g_k\in \Hvec$.
	Let $\cyl^0$ denote the subset of such functions which satisfy $g_1,\dots, g_k\in \V$.
	We denote by $\Phi'$ the differential of $\Phi$ in $\hdiv$, which can be expressed as
	\begin{equation*}
	\Phi'(u)= \sum_{j=1}^k \partial_j \phi\bigl((u,g_1),\dots,(u,g_k)\bigr) g_j,
	\end{equation*}
	where $\partial_j \phi$ is the derivative of $\phi$ with respect to its $j$th component.
\end{notation}
We can now define statistical solutions in the sense of Foia\c{s} and Prodi. We will use the definition as it stated in their newer work~\cite[Def. 3.2]{Foias2013}:

\begin{definition}[Foia\c{s}--Prodi \cite{FoiasTemam2001,Foias1972,FoiasProdi1976,Foias2013}]\label{def:foiastemamstatsol}

	A family of probability measures $(\mu_t)_{0\leq t\leq T}$ on $\hdiv$ is a \emph{Foia\c{s}--Prodi statistical solution} of the Navier--Stokes equations on $\hdiv$ with initial data $\mu_0$
	if
	\begin{enumerate}[{\bf (a)}]
		\item The function
		\begin{equation}\label{eq:measurabilitymu}
		t\mapsto \int_{L^2_{\Div}} \varphi(u)\,d\mu_t(u),
		\end{equation}
		is measurable on $[0,T]$ for every $\varphi\in C_b(\hdiv)$;
		\item $\mu$ satisfies the weak formulation
	\begin{equation}\label{eq:statsolFT}
		\int_{L^2_{\Div}} \Phi(u) \,d\mu_t(u)=\int_{L^2_{\Div}}\Phi(u) \,d\mu_0(u) + \int_0^t\int_{L^2_{\Div}} (F(s,u),\Phi'(u)) \,d\mu_s(u) \,ds
		\end{equation}
		for all $t\in[0,T]$ and all cylindrical test functions $\Phi\in \cyl^0$, where $F$ is given in \eqref{eq:F}.
		\item $\mu$ satisfies the strengthened mean energy inequality:
		 For any $\psi\in C^1(\R,\R)$ nonnegative, nondecreasing with bounded derivative and $t\in [0,T]$, the inequality
		\begin{multline}
		\label{eq:energystatstrong}
		\int_{L^2_{\Div}} \psi(\|u\|_{L^2(\dom)}^2) \,d\mu_t(u) + 2\vis \int_0^t\int_{L^2_{\Div}} \psi'(\|u\|_{L^2(\dom)}^2)  |u|_{H^1(\dom)}^2 \,d\mu_s(u)\\
		\leq
		\int_{L^2_{\Div}} \psi(\|u\|_{L^2(\dom)}^2 )\,d\mu_{0}(u)
		\end{multline}
		holds. 
		\item The function
		\begin{equation}
		\label{eq:continuityatzero}
		t\mapsto \int_{L^2_{\Div}}\psi(\norm{u}_{L^2(\dom)}^2)\, d\mu_t(u)
		\end{equation}
		is continuous at $t=0$ from the right, for any function $\psi\in C^1(\R,\R)$ nonnegative, nondecreasing with bounded derivative.
	\end{enumerate}
\end{definition}
\begin{remark}\label{rem:integrability}
Note that, as a consequence of the energy inequality \eqref{eq:energystatstrong} for $\psi(s)=s$, the function
\begin{equation*}
t\mapsto\int_{L^2_{\Div}}\|u\|_{L^2(\dom)}^2 \,d\mu_t(u),
\end{equation*}
belongs to $L^\infty([0,T])$ and the function
\begin{equation}\label{eq:h1linf}
t\mapsto\int_{L^2_{\Div}}|u|_{H^1(\dom)}^2 \,d\mu_t(u),
\end{equation}
belongs to $L^1([0,T])$. Notice also that~\eqref{eq:statsolFT} implies that
\begin{equation*}
t\mapsto \int_{L^2_{\Div}}\Phi(u)d\mu_t(u)
\end{equation*}
for $\Phi(u)$ a cylindrical test function, is continuous since
\begin{equation*}
\int_{L^2_{\Div}} (F(s,u),\Phi'(u)) \,d\mu_t(u)
\end{equation*}
is locally integrable. Combining this fact with condition~{\bf (c)}, conditionn~{\bf (d)} follows directly.
\end{remark}
\subsection{Equivalence between the solution concepts}
Next, we show that the Friedman--Keller statistical solutions in Definition~\ref{def:fksys} and the Foia\c{s}--Prodi statistical solutions in Definition~\ref{def:foiastemamstatsol} are in fact the same.
\begin{theorem}[Foia\c{s}--Temam statistical solutions satisfy the Friedman--Keller system]\label{thm:FT2FKsys_s}
	Let $\mu$ be a Foia\c{s}--Prodi statistical solution such that the initial condition $\mu_0$ has bounded support,
\[\text{supp}(\mu_0)\subset B\subset \hdiv, \qquad B\subset \big\{u\in \Lvec\ :\ \|u\|_{L^2(\dom)}\leq R\big\}
\]
 for some $R>0$. Then $\mu$ corresponds (cf.~Theorem~\ref{thm:equivalencemeasures}) to a correlation measure $\cm_t=(\nu^1,\nu^2,\dots)$ that is a statistical solution in the Friedman--Keller sense (cf.~Definition~\ref{def:fksys}). 
\end{theorem}
Conversely, we have:
\begin{theorem}[Friedman--Keller solutions are Foia\c{s}--Prodi statistical solutions]\label{thm:FKsys2FT_s}
		Let $\cm=(\cm_t)_{0\leq t\leq T}$ be a Friedman--Keller statistical solution of Navier--Stokes (cf.\ Definition \ref{def:fksys}) with bounded support, i.e.,
	\begin{equation}\label{eq:bdmoments}
	\int_{\dom^k}\int_{U^{k}} |\xi_1|^2\dots|\xi_k|^2 \,d\nu_{t,x}^k(\xi) \,dx\leq R^k<\infty,
	\end{equation}
	for some $0<R<\infty$, every $k\in \N$, and almost every $t\in [0,T]$. Then $\cm$ corresponds to a probability measure $\mu=(\mu_t)_{0\leq t\leq T}$ on a bounded set of $\hdiv$ which is a Foia\c{s}--Prodi statistical solution of the Navier--Stokes equations (cf.~Definition~\ref{def:foiastemamstatsol}).
\end{theorem}
The proofs of these two results are given in Appendix~\ref{app:equivalence}.

Foia\c{s} et al.\ have shown existence of Foia\c{s}--Prodi statistical solutions for the (forced) Navier--Stokes equations, see e.g.~\cite{Foias1972,Foias1973,FoiasTemam2001}. Using these equivalence theorems, this implies existence of statistical solutions via correlation measures as in Definition~\ref{def:fksys}.
\begin{remark}
	The equivalence theorems~\ref{thm:FT2FKsys_s} and~\ref{thm:FKsys2FT_s} are restricted to probability measures with bounded support. It should be possible to extend these results to probability measures having sufficiently fast decay near infinity; however, the proofs would become significantly more technical. We have therefore decided to restrict ourselves to probability measures with bounded support.
\end{remark}

\section{Vanishing viscosity limit of statistical solutions of Navier--Stokes}\label{sec:VV}
The goal of this section is to pass to the inviscid limit $\vis\to 0$ under the assumption of \emph{weak statistical scaling} (c.f. Section~\ref{sec:scaling}, Assumption~\ref{ass:scaling}). We will first prove a rigorous result on the longitudinal third order structure function
\begin{equation}\label{eq:S3def}
S_\|^3(\tau,r)\coloneqq\int_0^\tau\int_{\Hn}\!\!\fint_{\mathbb{S}^2}\int_{\dom}\big((u(x+r n)-u(x))\cdot n\big)^3 \,dxdS(n) \,d\mu_t(u)\,dt,
\end{equation}
and then relate it to the similarly defined second order structure function using the weak scaling assumption. Together with  weak statistical anisotropy, this yields diagonal continuity of the correlation measures $\cm^\vis$ that is needed to apply the compactness theorem~\ref{thm:timedepcmcompactness} and pass to the limit. The proof of the scaling estimate for the third order structure function~\eqref{eq:S3def} in Lemma~\ref{lem:S30} and ~\ref{lem:boundS3parallel} largely follows the proof of a similar result for martingale solutions of stochastic Navier--Stokes equations in~\cite{Bedrossian2019}. To simplify notation, we will omit writing the dependence of $\cm$ and $\mu$ on $\vis$ in the following sections.

\subsection{K\'arm\'an--Howarth--Monin relation}\label{sec:KHM}
The key to deriving an estimate on the behavior of the  third order structure function~\eqref{eq:S3def} is the so-called K\'arm\'an--Howarth--Monin (KHM) relation~\cite{deKarman1938} that describes the evolution of the second correlation marginal. Similar relations have been derived before for various settings (stochastic, forced, etc.), see~\cite{Frisch,Nie1999,deKarman1938,MoninYaglom,Bedrossian2019,Eyink_2002,DR00}. For statistical solutions we derive:
\begin{proposition}
	\label{prop:khm}
	Let $\cm$ be a Friedman--Keller statistical solution of the Navier--Stokes equations. Then the second correlation marginal $\nu^{2}$ satisfies the K\'arm\'an--Howarth--Monin relation for correlation measures:
	\begin{equation}	\label{eq:khm3}
	\begin{split}
	&\quad\sum_{ij}\int_\dom\!\int_\dom\!\int_{U^2}\!\!\xi_1^i \xi_2^j\,d\nu_{\tau,x,x+h}^2(\xi) \,dx\, \sigma^{ij}(h) \,dh
	-\sum_{ij}\int_{\dom}\!\int_\dom\!\int_{U^2}\!\!\xi_1^i \xi_2^j\,d\nu_{0,x,x+h}^2(\xi) \,dx \,\sigma^{ij}(h) \,dh \\
	&\quad+\frac12\sum_{ijk}\int_0^\tau\!\!\int_\dom\!\int_\dom\!\int_{U^2}(\xi_2^i-\xi^i_1)(\xi_2^j-\xi_1^j)(\xi_2^k-\xi^k_1)\,d\nu_{t,x,x+h}^2(\xi) \,dx \,\partial_{h^k}\sigma^{ij}(h)\,dh \,dt\\
	&=-\vis\sum_{ij}\int_0^\tau\!\!\int_\dom\!\int_\dom\!\int_{U^2}\!\!(\xi_1^i-\xi_2^i)(\xi_1^j- \xi_2^j)\,d\nu_{t,x,x+h}^2(\xi) \,dx\,\Delta_h\sigma^{ij}(h) \,dh \,dt.
	\end{split}
	\end{equation}
	for any $\tau>0$, where $\sigma=(\sigma_{ij})_{i,j=1}^3$ is any smooth, compactly supported, isotropic rank 2 tensor -- that is, any $\sigma\in C^2_c(\R^d, \R^{d\times d})$ of the form
		\begin{equation}\label{eq:sigma}
	\sigma(h)= \left(\omega_1(|h|) \mathbf{I}+ \omega_2(|h|)\hat{h}\otimes \hat{h}\right) \qquad \text{where } \hat{h}=\begin{cases}h/|h| & h\neq0 \\ 0 & h=0\end{cases}
	\end{equation}
	for $\omega_1,\omega_2\in C^2_c(\R)$.
\end{proposition}
\begin{proof}
	We consider equation~\eqref{eq:FKsyspressure} for $k=2$ with the test function $\varphi(t,x,y)= \eta(t,y-x)$ 
	 (for simplicity replacing $x_1$ and $x_2$ by $x$ and $y$ and writing in component form $\eta = (\eta^{ij})_{ij}$),
\begin{equation}\label{eq:marginal2}
\begin{split}
&\sum_{ij}\int_0^T\!\!\int_\dom\!\int_\dom\int_{U^2}\xi_1^i \xi_2^j\frac{\partial \eta^{ij}}{\partial t}(t,y-x)\,d\nu_{t,x,y}^2(\xi) \,dx\,dy \,dt\\
&\quad+\sum_{ij}\int_\dom\!\int_\dom\int_{U^2}\!\!\xi_1^i \xi_2^j\eta^{ij}(0,y-x)\,d\nu_{0,x,y}^2(\xi) \,dx \,dy \\
&\quad+\sum_{ijk}\int_0^T\!\!\int_\dom\!\int_\dom\int_{U^2}\xi_1^i\xi_1^k\xi_2^j\partial_{x^k}\eta^{ij}(t,y-x)\,d\nu_{t,x,y}^2(\xi) \,dx\,dy \,dt\\
&\quad+\sum_{ijk}\int_0^T\!\!\int_\dom\!\int_\dom\int_{U^2}\xi_1^i\xi_2^j\xi_2^k\partial_{y^k}\eta^{ij}(t,y-x)\,d\nu_{t,x,y}^2(\xi) \,dx\,dy \,dt\\
&=-\vis\sum_{ij}\int_0^T\!\!\int_\dom\!\int_\dom\int_{U^2}\!\!\xi_1^i \xi_2^j(\Delta_{x}+\Delta_y)\eta^{ij}(t,y-x)\,d\nu_{t,x,y}^2(\xi) \,dx \,dy \,dt\\
&\quad+\sum_{ijk}\int_0^T\!\!\int_\dom\!\int_\dom\int_{U^2}\xi_1^i\xi_1^k\xi_2^j\partial_{x^k}\partial_{x^i}\psi^j_{\eta,1}(t,y-x)\,d\nu_{t,x,y}^2(\xi) \,dx \,dy \,dt\\
&\quad+\sum_{ijk}\int_0^T\!\!\int_\dom\!\int_\dom\int_{U^2}\xi_1^i\xi_2^j\xi_2^k\partial_{y^k}\partial_{y^j}\psi^i_{\eta,2}(t,y-x)\,d\nu_{t,x,y}^2(\xi) \,dx \,dy \,dt
\end{split}
\end{equation}
Since $\psi_{\eta,1}$ solves $\Delta_x\psi_{\eta,1}=\Div_x\eta$ and $\psi_{\eta,2}$ solves $\Delta_y\psi_{\eta,2}=\Div_y\eta=-\Div_x\eta$, we have $\psi_{\eta,1}= (\Delta_x)^{-1}\Div_x\eta$ and $\psi_{\eta,2}=-(\Delta_y)^{-1}\Div_x\eta=-(\Delta_x)^{-1}\Div_x\eta$ and so $\psi_\eta\coloneqq\psi_{\eta,1}=-\psi_{\eta,2}$ (up to additive constants). Using this and changing the integration variables to $x$ and $h\coloneqq y-x$, we obtain
\begin{equation}\label{eq:marginal2h}
\begin{split}
&\sum_{ij}\int_0^T\int_\dom\!\!\int_\dom\int_{U^2}\xi_1^i \xi_2^j\frac{\partial\eta^{ij}}{\partial t}(t,h)\,d\nu_{t,x,x+h}^2(\xi) \,dx\,dh \,dt\\
&\quad+\sum_{ij}\int_\dom\!\!\int_\dom\int_{U^2}\!\!\xi_1^i \xi_2^j\eta^{ij}(0,h)\,d\nu_{0,x,x+h}^2(\xi) \,dx \,dh \\
&\quad-\sum_{ijk}\int_0^T\!\!\int_\dom\!\!\int_\dom\int_{U^2}\xi_1^i\xi_1^k\xi_2^j\partial_{h^k}\eta^{ij}(t,h)\,d\nu_{t,x,x+h}^2(\xi) \,dx\,dh \,dt\\
&\quad+\sum_{ijk}\int_0^T\!\!\int_\dom\!\!\int_\dom\int_{U^2}\xi_1^i\xi_2^j\xi_2^k\partial_{h^k}\eta^{ij}(t,h)\,d\nu_{t,x,x+h}^2(\xi) \,dx\,dh \,dt\\
&=-2\vis\sum_{ij}\int_0^T\!\!\int_\dom\!\!\int_\dom\int_{U^2}\!\!\xi_1^i \xi_2^j\Delta_h\eta^{ij}(t,h)\,d\nu_{t,x,x+h}^2(\xi) \,dx \,dh \,dt\\
&\quad+\sum_{ijk}\int_0^T\!\!\int_\dom\!\!\int_\dom\int_{U^2}\xi_1^i\xi_1^k\xi_2^j\partial_{h^k}\partial_{h^i}\psi^j_{\eta}(t,h)\,d\nu_{t,x,x+h}^2(\xi) \,dx \,dh \,dt\\
&\quad-\sum_{ijk}\int_0^T\!\!\int_\dom\!\!\int_\dom\int_{U^2}\xi_1^i\xi_2^j\xi_2^k\partial_{h^k}\partial_{h^j}\psi^i_{\eta}(t,h)\,d\nu_{t,x,x+h}^2(\xi) \,dx \,dh \,dt
\end{split}
\end{equation}
where $\psi_\eta = \psi_{\eta,1}$.
The cubic terms can be rewritten using the following simple fact (which can also be found in Frisch~\cite[equation (6.13)]{Frisch}, and in a similar, weak form in~\cite{Bedrossian2019}):
\begin{equation}\label{eq:stupididentity}
\begin{split}
	&\quad-2\sum_{ijk}\int_0^T\!\!\int_\dom\!\!\int_\dom\int_{U^2}\xi_1^i\xi_2^j(\xi_1^k-\xi^k_2)\partial_{h^k}\eta^{ij}(t,h)\,d\nu_{t,x,x+h}^2(\xi) \,dx\,dh \,dt\\
	&=\sum_{ijk}\int_0^T\!\!\int_\dom\!\!\int_\dom\int_{U^2}(\xi_1^i-\xi^i_2)(\xi_1^j-\xi_2^j)(\xi_1^k-\xi^k_2)\partial_{h^k}\eta^{ij}(t,h)\,d\nu_{t,x,x+h}^2(\xi) \,dx\,dh \,dt.
\end{split}
\end{equation}
The proof of this is postponed to the end of this proof. Using this, we can rewrite~\eqref{eq:marginal2h} as
\begin{equation}\label{eq:khm}
\begin{split}
&\quad\sum_{ij}\int_0^T\!\!\int_\dom\!\!\int_\dom\int_{U^2}\!\!\xi_1^i \xi_2^j\frac{\partial \eta^{ij}}{\partial t}(t,h)\,d\nu_{t,x,x+h}^2(\xi) \,dx\,dh \,dt\\
&\quad+\sum_{ij}\int_\dom\!\!\int_\dom\int_{U^2}\!\!\xi_1^i \xi_2^j\eta^{ij}(0,h)\,d\nu_{0,x,x+h}^2(\xi) \,dx \,dh \\
&\quad-\frac12\sum_{ijk}\int_0^T\!\!\int_\dom\!\!\int_\dom\int_{U^2}(\xi_1^i-\xi^i_2)(\xi_1^j-\xi_2^j)(\xi_1^k-\xi^k_2)\partial_{h^k}\eta^{ij}(t,h)\,d\nu_{t,x,x+h}^2(\xi) \,dx\,dh \,dt\\
&=-2\vis\sum_{ij}\int_0^T\!\!\int_\dom\!\!\int_\dom\int_{U^2}\!\!\xi_1^i \xi_2^j\Delta_h\eta^{ij}(t,h)\,d\nu_{t,x,x+h}^2(\xi) \,dx \,dh \,dt\\
&\quad+\sum_{ijk}\int_0^T\!\!\int_\dom\!\!\int_\dom\int_{U^2}\xi_1^i\xi_1^k\xi_2^j\partial_{h^k}\partial_{h^i}\psi^j_{\eta}(t,h)\,d\nu_{t,x,x+h}^2(\xi) \,dx \,dh \,dt\\
&\quad-\sum_{ijk}\int_0^T\!\!\int_\dom\!\!\int_\dom\int_{U^2}\xi_1^i\xi_2^j\xi_2^k\partial_{h^k}\partial_{h^j}\psi^i_{\eta}(t,h)\,d\nu_{t,x,x+h}^2(\xi) \,dx \,dh \,dt.
\end{split}
\end{equation}
Since
\begin{align*}
&\quad\sum_{ij}\int_0^T\!\!\int_\dom\!\!\int_\dom\int_{U^2}\!\!\xi_1^i \xi_1^j\Delta_h\eta^{ij}(t,h)\,d\nu_{t,x,x+h}^2(\xi) \,dx \,dh \,dt\\
&= \sum_{ij}\int_0^T\!\!\int_{\dom}\!\int_{U}\!\!\xi_1^i \xi_1^j \,d\nu_{t,x}^1(\xi) \,dx \int_{\dom} \Delta_h\eta^{ij}(t,h)\,dh \,dt=0,
\end{align*}
we can rewrite
\begin{align*}
&\quad-2\vis\sum_{ij}\int_0^T\!\!\int_\dom\!\!\int_\dom\int_{U^2}\!\!\xi_1^i \xi_2^j\Delta_h\eta^{ij}(t,h)\,d\nu_{t,x,x+h}^2(\xi) \,dx \,dh \,dt\\
&=\vis\sum_{ij}\int_0^T\!\!\int_\dom\!\!\int_\dom\int_{U^2}\!\!(\xi_1^i-\xi_2^i)(\xi_1^j- \xi_2^j)\Delta_h\eta^{ij}(t,h)\,d\nu_{t,x,x+h}^2(\xi) \,dx \,dh \,dt.
\end{align*}
Moreover, we have for symmetric, smooth and compactly supported rank 2 tensors $\eta$ of the form
\begin{equation}\label{eq:eta1}
\eta(t,h)= \omega_1(t,|h|) \mathbf{I}+ \omega_2(t,|h|)\hat{h}\otimes \hat{h},
\end{equation}
with $\psi^i_\eta = -(\Delta_h)^{-1}\Div_h\eta^{i,\cdot}$,
\begin{equation}\label{eq:stupididentity2}
\begin{split}
	&\sum_{ijk}\int_0^T\!\!\int_\dom\!\!\int_\dom\int_{U^2}\xi_1^i\xi_1^k\xi_2^j\partial_{h^k}\partial_{h^i}\psi^j_{\eta}(t,h)\,d\nu_{t,x,x+h}^2(\xi) \,dx \,dh \,dt = 0,\\
	&\sum_{ijk}\int_0^T\!\!\int_\dom\!\!\int_\dom\int_{U^2}\xi_1^i\xi_2^j\xi_2^k\partial_{h^k}\partial_{h^j}\psi^i_{\eta}(t,h)\,d\nu_{t,x,x+h}^2(\xi) \,dx \,dh \,dt=0,
\end{split}
\end{equation}
whose proof is postponed to the end of this proof. Using this,~\eqref{eq:khm} becomes
\begin{equation}\label{eq:khm2}
\begin{split}
&\sum_{ij}\int_0^T\!\!\int_\dom\!\!\int_\dom\int_{U^2}\!\!\xi_1^i \xi_2^j\frac{\partial \eta^{ij}}{\partial t}(t,h)\,d\nu_{t,x,x+h}^2(\xi) \,dx\,dh \,dt\\
&\quad+\sum_{ij}\int_\dom\!\!\int_\dom\int_{U^2}\!\!\xi_1^i \xi_2^j\eta^{ij}(0,h)\,d\nu_{0,x,x+h}^2(\xi) \,dx \,dh \\
&\quad-\frac12\sum_{ijk}\int_0^T\!\!\int_\dom\!\!\int_\dom\int_{U^2}(\xi_1^i-\xi^i_2)(\xi_1^j-\xi_2^j)(\xi_1^k-\xi^k_2)\partial_{h^k}\eta^{ij}(t,h)\,d\nu_{t,x,x+h}^2(\xi) \,dx\,dh \,dt\\
&=\vis\sum_{ij}\int_0^T\!\!\int_\dom\!\!\int_\dom\int_{U^2}\!\!(\xi_1^i-\xi_2^i)(\xi_1^j- \xi_2^j)\Delta_h\eta^{ij}(t,h)\,d\nu_{t,x,x+h}^2(\xi) \,dx \,dh \,dt.
\end{split}
\end{equation}
Let $\theta_\delta$ be a sequence of smooth, uniformly bounded functions with the property that $\theta_\delta\to \mathbf{1}_{(0,\tau]}(t)$ for every $t$ as $\delta\to 0$. If we now use a test function
\begin{equation}\label{eq:eta2}
\eta(t,h)= \sigma(h)\theta_\delta(t),
\end{equation}
where $\sigma$ is of the form \eqref{eq:sigma}, then we can use the weak continuity in time of the moments $\int_{U^k}\xi_1\otimes\dots\otimes \xi_k \,d\nu^k_{t,x}(\xi)$ to obtain for any $\tau>0$, as $\delta\to 0$,
\begin{equation}
\begin{split}
&\quad-\sum_{ij}\int_\dom\!\!\int_\dom\int_{U^2}\!\!\xi_1^i \xi_2^j\,d\nu_{\tau,x,x+h}^2(\xi) \,dx  \sigma^{ij}(h) \,dh +\sum_{ij}\int_\dom\!\!\int_\dom\int_{U^2}\!\!\xi_1^i \xi_2^j\,d\nu_{0,x,x+h}^2(\xi) \,dx \sigma^{ij}(h) \,dh \\
&\quad-\frac12\sum_{ijk}\int_0^\tau\!\!\int_\dom\!\!\int_\dom\int_{U^2}(\xi_1^i-\xi^i_2)(\xi_1^j-\xi_2^j)(\xi_1^k-\xi^k_2)\,d\nu_{t,x,x+h}^2(\xi) \,dx \partial_{h^k}\sigma^{ij}(h)\,dh \,dt\\
&=\vis\sum_{ij}\int_0^\tau\!\!\int_\dom\!\!\int_\dom\int_{U^2}\!\!(\xi_1^i-\xi_2^i)(\xi_1^j- \xi_2^j)\,d\nu_{t,x,x+h}^2(\xi) \,dx\Delta_h\sigma^{ij}(h) \,dh \,dt.\qedhere
\end{split}
\end{equation}
\end{proof}
\begin{proof}[Proof of \eqref{eq:stupididentity}]
	We expand the right hand side:
	\begin{align*}
	&\sum_{ijk}\int_0^T\!\!\int_\dom\!\!\int_\dom\int_{U^2}(\xi_1^i-\xi^i_2)(\xi_1^j-\xi_2^j)(\xi_1^k-\xi^k_2)\partial_{h^k}\eta^{ij}(t,h)\,d\nu_{t,x,x+h}^2(\xi) \,dx\,dh \,dt\\
	&= \sum_{ijk}\int_0^T\!\!\int_\dom\!\!\int_\dom\int_{U^2}(\xi_1^i\xi_1^j\xi_1^k-\xi_2^i\xi_2^j\xi_2^k)\,d\nu_{t,x,x+h}^2(\xi) \,dx \,\partial_{h^k}\eta^{ij}(t,h)\,dh \,dt\\
	&\quad +
	\sum_{ijk}\int_0^T\!\!\int_\dom\!\!\int_\dom\int_{U^2}(\xi_2^i\xi_2^j\xi_1^k-\xi_1^i\xi_1^j\xi_2^k)\partial_{h^k}\eta^{ij}(t,h)\,d\nu_{t,x,x+h}^2(\xi) \,dx\,dh \,dt\\
	&\quad -
	\sum_{ijk}\int_0^T\!\!\int_\dom\!\!\int_\dom\int_{U^2}(\xi_1^i\xi_2^j\xi_1^k-\xi_2^i\xi_1^j\xi_2^k)\partial_{h^k}\eta^{ij}(t,h)\,d\nu_{t,x,x+h}^2(\xi) \,dx\,dh \,dt\\
	&\quad - \sum_{ijk}\int_0^T\!\!\int_\dom\!\!\int_\dom\int_{U^2}(\xi_2^i\xi_1^j\xi_1^k-\xi_1^i\xi_2^j\xi_2^k)\partial_{h^k}\eta^{ij}(t,h)\,d\nu_{t,x,x+h}^2(\xi) \,dx\,dh \,dt.
	\end{align*}
	The first term on the right hand side is zero since $\eta$ is compactly supported (after changing the integration variable from $x$ to $x-h$ in one of the terms). The second term on the right hand side vanishes using the divergence constraint~\eqref{eq:divfreeconst}. Using that $\eta$ and $\nu$ are symmetric, the last two terms are identical and so
	\begin{multline*}
	\sum_{ijk}\int_0^T\!\!\int_\dom\!\!\int_\dom\int_{U^2}(\xi_1^i-\xi^i_2)(\xi_1^j-\xi_2^j)(\xi_1^k-\xi^k_2)\partial_{h^k}\eta^{ij}(t,h)\,d\nu_{t,x,x+h}^2(\xi) \,dx\,dh \,dt\\
	= -2\sum_{ijk}\int_0^T\!\!\int_\dom\!\!\int_\dom\int_{U^2}(\xi_2^i\xi_1^j\xi_1^k-\xi_1^i\xi_2^j\xi_2^k)\partial_{h^k}\eta^{ij}(t,h)\,d\nu_{t,x,x+h}^2(\xi) \,dx\,dh \,dt,
	\end{multline*}
	which proves the claim.
\end{proof}
\begin{proof}[Proof of \eqref{eq:stupididentity2}]
	We consider the second expression,
assume $\eta$ is of the form
	\begin{equation}\label{eq:eta1}
	\eta(t,h)= \omega_1(t,|h|) \mathbf{I}+ \omega_2(t,|h|)\hat{h}\otimes \hat{h},
	\end{equation}
	where $\omega_i$, $i=1,2$ are compactly supported in the torus. Then using that $\psi^i_\eta = -(\Delta_h)^{-1}\Div_h\eta^{i,\cdot}$ (the first term is treated in a similar way)
	\begin{align*}
	E &\coloneqq -\sum_{ijk}\int_0^T\!\!\int_\dom\!\!\int_\dom\int_{U^2}\xi_1^i\xi_2^j\xi_2^k\partial_{h^k}\partial_{h^j}\psi^i_{\eta}(t,h)\,d\nu_{t,x,x+h}^2(\xi) \,dx \,dh \,dt\\
	&= \sum_{ijk\ell}\int_0^T\!\!\int_\dom\!\!\int_\dom\int_{U^2}\xi_1^i\xi_2^j\xi_2^k\,d\nu_{t,x,x+h}^2(\xi) \,dx\,\partial_{h^k}\partial_{h^j}\Delta^{-1}\left(\partial_{h^\ell}\eta^{i\ell}(t,h)\right) \,dh \,dt\\
	&= \sum_{ijk\ell}\int_0^T\!\!\int_{\dom}\!\partial_{h^k}\partial_{h^j}\Delta^{-1}\left(\int_\dom\int_{U^2}\xi_1^i\xi_2^j\xi_2^k\,d\nu_{t,x,x+h}^2(\xi) \,dx\right)\partial_{h^\ell}\eta^{i\ell}(t,h)dh \,dt.
	\end{align*}
	We note that since $\omega_i$ have compact support, we can write in polar coordinates
	\begin{equation*}
	\sum_\ell \partial_{h^\ell} \eta^{i\ell}(t,h) = \biggl(\underbrace{\omega'_1(t,|h|)+\omega_2'(t,|h|)+2\frac{\omega_2(t,|h|)}{|h|}}_{\eqqcolon\,G(t,|h|)}\biggr)\hat{h}^i
	\end{equation*}
	and so
	\begin{equation*}
	\begin{split}
	E&=\sum_{ijk\ell}\int_0^T\!\!\int_{\dom}\!\partial_{h^k}\partial_{h^j}\Delta^{-1}\left(\int_\dom\int_{U^2}\xi_1^i\xi_2^j\xi_2^k\,d\nu_{t,x,x+h}^2(\xi) \partial_{h^\ell}\eta^{i\ell}(t,h)\,dx\right)dh \,dt\\
	&= \sum_{ijk}\int_0^T\!\!\int_{\dom}\!\partial_{h^k}\partial_{h^j}\Delta^{-1}\left(\int_\dom\int_{U^2}\xi_1^i\xi_2^j\xi_2^k\,d\nu_{t,x,x+h}^2(\xi) \,dx\right) G(t,|h|)\hat{h}^i \,dh \,dt\\
	&= \sum_{jk}\int_0^T\!\!\int_0^\infty\int_{|h|=r}\!\partial_{h^k}\partial_{h^j}\Delta^{-1}\left(\int_\dom\!\int_{U^2}\xi_1^i\xi_2^j\xi_2^k\,d\nu_{t,x,x+h}^2(\xi) \,dx\right) \hat{h}^i dS(h)\, G(t,r)\, dr \,dt\\
	&= \sum_{ijk}\int_0^T\!\!\int_0^\infty\int_{|h|\leq r}\!\partial_{h^i}\partial_{h^k}\partial_{h^j}\Delta^{-1}\left(\int_\dom\!\int_{U^2}\xi_1^i\xi_2^j\xi_2^k\,d\nu_{t,x,x+h}^2(\xi) \,dx\right)dh\, G(t,r)\, dr \,dt\\
	&= \sum_{jk}\int_0^T\!\!\int_0^\infty\int_{|h|\leq r}\!\partial_{h^k}\partial_{h^j}\Delta^{-1}\Div_h\left(\int_\dom\!\int_{U^2}\xi_1\xi_2^j\xi_2^k\,d\nu_{t,x,x+h}^2(\xi) \,dx\right)dh\, G(t,r)\, dr \,dt=0,
	\end{split}
	\end{equation*}
	where we used the divergence theorem in the second to last identity and the divergence constraint~\eqref{eq:divfreeconst} for the last identity.
\end{proof}
\subsection{Scaling of third order structure functions}\label{sec:ODEs}
Next, we use the KHM-relation~\eqref{eq:khm3} to derive a scaling relation for the averaged third order structure function in terms of the measure $\mu_t$,
\begin{equation}\label{eq:s03}
S^{3}_{0}(\tau,r)=\int_0^\tau\int_{\Hn}\!\fint_{\mathbb{S}^2}\int_{\dom}\big|u(x)-u(x+r n)\big|^2\big(u(x+r n)-u(x)\big)\cdot n  \,dx\, dS(n) \,d\mu_t(u)\,dt,
\end{equation}
which will be more convenient to work with for this purpose. We have:
\begin{lemma}
	\label{lem:S30}
	Let $\mu_t$ be a Foia\c{s}--Prodi statistical solution of the Navier--Stokes equations (cf.\ Definition~\ref{def:foiastemamstatsol}). Then
	\begin{equation}
	\left|\frac{S^{3}_{0}(\tau,r)}{r}\right|\leq 2 E_0,
	\end{equation}
	where $E_0$ is the initial energy,
	\begin{equation}\label{eq:E0definition}
	E_0\coloneqq \int_{\hdiv}\norm{u(x)}^2_{L^2(\dom)} \,d\mu_0(u).
	\end{equation}
\end{lemma}
\begin{proof}
	We take a test function of the form $\sigma(h)=\omega(|h|)\mathbf{I}$ in the KHM-relation~\eqref{eq:khm3} with $\omega$ having compact support in $[0,0.5)$. A little bit of algebra yields (denoting $\hat{h}_{k}=h^k/|h|$)
	\begin{equation*}
	\partial_{h^k}\omega(|h|)=\omega'(|h|)\hat{h}^k,
	\end{equation*}
	and so~\eqref{eq:khm3} for this particular test function reads
	\begin{equation}\label{eq:khmid}
	\begin{split}
	&\int_\dom\int_\dom\!\int_{U^2}\!\!\xi_1\cdot\xi_2 \,d\nu_{\tau,x,x+h}^2(\xi) \,dx  \omega(|h|) \,dh
	-\int_\dom\int_\dom\!\int_{U^2}\!\!\xi_1\cdot \xi_2\,d\nu_{0,x,x+h}^2(\xi) \,dx \omega(|h|) \,dh \\
	&+\frac12\int_0^\tau\!\!\int_\dom\int_\dom\!\int_{U^2}|\xi_1-\xi_2|^2(\xi_2-\xi_1)\cdot \hat{h}\,d\nu_{t,x,x+h}^2(\xi) \,dx \omega'(|h|)\,dh \,dt\\
	=&-\vis\int_0^\tau\!\!\int_\dom\int_\dom\!\int_{U^2}\!\!|\xi_1-\xi_2|^2 \,d\nu_{t,x,x+h}^2(\xi) \,dx\Delta_h\omega(|h|) \,dh \,dt.
	\end{split}
	\end{equation}
	In terms of the statistical solution $(\mu_t)_{t>0}$ this is
	\begin{equation}	\label{eq:khmidu}
	\begin{split}
	&\quad\int_{\Hn}\!\int_\dom\int_\dom\!\! u(x)\cdot u(x+h) \,dx  \omega(|h|) \,dh \,d\mu_\tau(u)
	-\int_{\Hn}\int_\dom\int_\dom\! u(x)u(x+h) \,dx\, \omega(|h|) \,dh \,d\mu_0(u) \\
	&\quad+\frac12\int_0^\tau\!\!\int_{\Hn}\int_\dom\int_\dom\!|u(x)-u(x+h)|^2(u(x+h)-u(x))\cdot \hat{h}\, \,dx \omega'(|h|)\,dh \,d\mu_t(u) \,dt\\
	&=-\vis\int_0^\tau\!\!\int_{\Hn}\int_\dom\int_\dom\!|u(x)-u(x+h)|^2 \,dx\Delta_h\omega(|h|) \,dh \,d\mu_t(u) \,dt.
	\end{split}
	\end{equation}
	The last term can also be written as
	\begin{equation}\label{eq:diffusioncomputation}
	\begin{split}
	&\vis\int_0^\tau\!\!\int_{\Hn}\int_\dom\int_\dom\!|u(x)-u(x+h)|^2 \,dx\Delta_h\omega(|h|) \,dh \,d\mu_t(u) \,dt\\
	& = 2\vis\int_0^\tau\!\!\int_{\Hn}\int_\dom\int_\dom\!\Grad_h u(x+h)\cdot(u(x)-u(x+h)) \,dx\Grad_h\omega(|h|) \,dh \,d\mu_t(u) \,dt\\
	& = 2\vis\int_0^\tau\!\!\int_{\Hn}\int_\dom\int_\dom\!\Grad_x u(x+h)\cdot(u(x)-u(x+h)) \,dx\Grad_h\omega(|h|) \,dh \,d\mu_t(u) \,dt\\
	& = 2\vis\int_0^\tau\!\!\int_{\Hn}\int_\dom\int_\dom\!\Grad_x u(x)\cdot(u(x-h)-u(x)) \,dx\Grad_h\omega(|h|) \,dh \,d\mu_t(u) \,dt\\
	& = -2\vis\int_0^\tau\!\!\int_{\Hn}\int_\dom\int_\dom\!\Grad_x u(x): \Grad_h u(x-h) \,dx \omega(|h|) \,dh \,d\mu_t(u) \,dt\\
	& = 2\vis\int_0^\tau\!\!\int_{\Hn}\int_\dom\int_\dom\!\Grad_x u(x+h): \Grad_x u(x) \,dx \omega(|h|) \,dh \,d\mu_t(u) \,dt
	\end{split}
	\end{equation}
Changing to spherical coordinates 
 and using the definition of $S^{3}_0$,~\eqref{eq:s03}, we obtain
\begin{align*}
&\frac12\int_0^\infty S_0^3(\tau,r) r^2\omega'(r)dr\\
&=-2\vis\int_0^\infty\int_0^\tau\!\!\int_{\Hn}\fint_{\mathbb{S}^2}\int_{\dom}\!\Grad_x u(x):\Grad_x u(x+r n) \,dxdS(n)\,d\mu_t(u) \,dt r^2\omega(r) dr \\
&\quad-\int_0^\infty\int_{\Hn}\!\fint_{\mathbb{S}^2}\int_{\dom}\! u(x)\cdot u(x+r n) \,dxdS(n) \,d\mu_\tau(u)r^2\omega(r) dr \\
&\quad+\int_0^\infty\int_{\Hn}\fint_{\mathbb{S}^2}\int_{\dom}\! u(x)\cdot u(x+r n) \,dx dS(n) \,d\mu_0(u) r^2\omega(r) dr.
\end{align*}
We denote
\begin{align*}
m_2(\tau,r)&\coloneqq \int_{\Hn}\!\fint_{\mathbb{S}^2}\int_{\dom}\! u(x)\cdot u(x+r n) \,dxdS(n) \,d\mu_\tau(u)\\
v(\tau,r)&\coloneqq \int_0^\tau\!\!\int_{\Hn}\fint_{\mathbb{S}^2}\int_{\dom}\!\Grad_x u(x):\Grad_x u(x+r n) \,dxdS(n)\,d\mu_t(u) \,dt.
\end{align*}
Since $(\mu_t)_{t>0}$ is supported on functions in $L^\infty([0,\infty);\hdiv)\cap L^2([0,\infty);\V)$, $S_0^3$ is a continuous function. Moreover, notice that due to the \emph{a priori} bounds following from the energy inequality~\eqref{eq:energystatstrong}, both $m_2$ and $v$ are uniformly bounded and continuous in $r$ and $\tau$ (for the continuity in $\tau$ of the first quantity, one needs weak time continuity of the moments which follows the fact that they satisfy the equations~\eqref{eq:khm3} where all the terms are integrable). We obtain
\begin{align*}
\frac12\int_0^\infty S_0^3(\tau,r) r^2\omega'(r)dr &=-2\vis\int_0^\infty v(\tau,r)r^2\omega(r) dr \\
&\quad-\int_0^\infty m_2(\tau,r)r^2\omega(r) dr +\int_0^\infty m_2(0,r) r^2\omega(r) dr,
\end{align*}
which is an ODE in the sense of distributions for $S_0^3(\tau,\cdot)$, and because the right hand side is uniformly bounded and continuous, we can consider it in the strong sense (note that boundary terms when integrating the $S_0^3$ term by parts vanish):
\begin{equation*}
\frac{1}{r^2}\partial_r(r^2 S_0^3(r)) = 4\eps v(\tau,r)+2 m_2(\tau,r)-2 m_2(0,r),
\end{equation*}
or
\begin{equation*}
\frac{S_0^3(r)}{r} = \frac{2}{r^3}\int_0^r s^2\big(2\eps v(\tau,s)+m_2(\tau,s)-m_2(0,s)\big) \,ds.
\end{equation*}
The energy inequality~\eqref{eq:energystatstrong} and the Cauchy--Schwarz inequality imply that $2\vis v(\tau,s)$ and $m_2(\tau,s)$ are both bounded by $E_0$ (defined in \eqref{eq:E0definition}), uniformly in $\tau,s$. Hence,
\begin{equation}\label{eq:S_0bound}
\left|\frac{S_0^3(r)}{r}\right|\leq \frac{2}{r}\int_0^r   \left(\frac{s}{r}\right)^2\big(2\eps|v(\tau,s)|+|m_2(\tau,s)|+|m_2(0,s)|\big) \,ds\leq 2E_0.
\end{equation}
(See also~\cite[Proposition 1.9]{Bedrossian2019} for a related result.)
\end{proof}

Using this lemma, we can derive a scaling relation for the averaged longitudinal structure function $S^3_\|$, where
\begin{equation}
\label{eq:Spparallel}
S_\|^p(\tau,r)=\int_0^\tau\int_{\Hn}\!\!\fint_{\mathbb{S}^2}\int_{\dom}\big((u(x+r n)-u(x))\cdot n\big)^p \,dxdS(n) \,d\mu_t(u)\,dt.
\end{equation}
\begin{lemma}
	\label{lem:boundS3parallel}
		Let $\mu_t$ be a Foia\c{s}--Prodi statistical solution of the Navier--Stokes equations (cf.\ Definition~\ref{def:foiastemamstatsol}). Then
	\begin{equation}
	\left|\frac{S^{3}_\|(\tau,r)}{r}\right|\leq 2 E_0,
	\end{equation}
	where $C>0$ is some constant independent of $\eps$ and $E_0$ is the initial energy,
	\begin{equation}
	E_0\coloneqq \int_{\hdiv}\norm{u(x)}^2_{L^2(\dom)} \,d\mu_0(u)
	\end{equation}
\end{lemma}
\begin{proof}
	Again, we start with the KHM relation~\eqref{eq:khm3}. This time we use the test function $\sigma(h) = \omega(|h|)\hat{h}\otimes \hat{h}$ where $\omega\in C_c^\infty(\R)$ is an even function. We have
	\begin{equation*}
	\begin{split}
	\partial_{h^k} (\omega(|h|)\hat{h}^i \hat{h}^j)& = \omega'(|h|)\hat{h}^i\hat{h}^j\hat{h}^k + \frac{1}{|h|}\left(\delta_{ik}\hat{h}^j+\delta_{jk}\hat{h}^i-2\hat{h}^i\hat{h}^j\hat{h}^k\right)\omega(|h|)\\
	& =\left(\omega'(|h|)-2\frac{\omega(|h|)}{|h|}\right)\hat{h}^i\hat{h}^j\hat{h}^k + \frac{1}{|h|}\left(\delta_{ik}\hat{h}^j+\delta_{jk}\hat{h}^i\right)\omega(|h|)
	\end{split}
	\end{equation*}
	Therefore,~\eqref{eq:khm3} becomes
	\begin{equation}\label{eq:khmhat}
	\begin{split}
	&-\int_\dom\int_\dom\!\int_{U^2}\!\!\xi_1\cdot\hat{h}\, \xi_2\cdot \hat{h} \,d\nu_{\tau,x,x+h}^2(\xi) \,dx  \omega(|h|) \,dh\\
	&	+\int_\dom\int_\dom\!\int_{U^2}\!\!\xi_1\cdot \hat{h} \,\xi_2 \cdot \hat{h} \, \,d\nu_{0,x,x+h}^2(\xi) \,dx \omega(|h|) \,dh \\
	&-\frac12\int_0^\tau\!\!\int_\dom\int_\dom\!\int_{U^2}\big((\xi_2-\xi_1)\cdot \hat{h}\big)^3\,d\nu_{t,x,x+h}^2(\xi) \,dx\left(\omega'(|h|)-2|h|^{-1}\omega(|h|)\right) \,dh \,dt\\
	&+\int_0^\tau\!\!\int_\dom\int_\dom\!\int_{U^2}|\xi_1-\xi_2|^2(\xi_1-\xi_2)\cdot \hat{h} \,d\nu_{t,x,x+h}^2(\xi) \,dx |h|^{-1}\omega(|h|) \,dh \,dt\\
	=&\ \vis\sum_{ij}\int_0^\tau\!\!\int_\dom\int_\dom\!\int_{U^2}\!\!(\xi_1^i-\xi_2^i)(\xi_1^j- \xi_2^j)\,d\nu_{t,x,x+h}^2(\xi) \,dx\Delta_h\sigma^{ij}(h) \,dh \,dt.
	\end{split}
	\end{equation}
	Again, in terms of $(\mu_t)_{t>0}$, this means
	\begin{equation}\label{eq:khmhatu}
	\begin{split}
	&\quad-\int_{\Hn}\int_\dom\int_\dom\! u(x)\cdot\hat{h}\, u(x+h)\cdot \hat{h} \, \,dx  \omega(|h|) \,dh \,d\mu_\tau(u)\\
	&\quad+\int_{\Hn}\int_\dom\int_\dom\!u(x)\cdot \hat{h} \,u(x+h) \cdot \hat{h} \,dx \omega(|h|) \,dh \,d\mu_0(u) \\
	&\quad-\frac12\int_0^\tau\int_\dom\int_{\Hn}\!\!\int_\dom\big((u(x+h)-u(x))\cdot \hat{h}\big)^3 \,dx \,d\mu_t(u)\left(\omega'(|h|)-2|h|^{-1}\omega(|h|)\right) \,dh \,dt\\
	&\quad+\int_0^\tau\int_\dom\int_{\Hn}\!\!\int_\dom|u(x)-u(x+h)|^2\big(u(x)-u(x+h)\big)\cdot \hat{h}  \,dx \,d\mu_t(u) |h|^{-1}\omega(|h|) \,dh \,dt\\
	&=\vis\sum_{ij}\int_0^\tau\int_\dom\int_{\Hn}\!\!\int_\dom\!\big(u^i(x)-u^i(x+h)\big)\big(u^j(x)- u^j(x+h)\big) \,dx \,d\mu_t(u)\Delta_h\sigma^{ij}(h) \,dh \,dt.
	\end{split}
	\end{equation}
	Similar to the computation in \eqref{eq:diffusioncomputation}, we have
	\begin{align*}
	&\quad\vis\sum_{ij}\int_0^\tau\int_\dom\int_{\Hn}\!\!\int_\dom\!\big(u^i(x)-u^i(x+h)\big)\big(u^j(x)- u^j(x+h)\big) \,dx \,d\mu_t(u)\Delta_h\sigma^{ij}(h) \,dh \,dt\\
	&=2\vis\sum_{ij}\int_0^\tau\int_\dom\int_{\Hn}\!\!\int_\dom\!\Grad_xu^i(x+h)\Grad_x u^j(x) \,dx \,d\mu_t(u)\sigma^{ij}(h) \,dh \,dt\\
	&=2\vis\int_0^\tau\int_\dom\int_{\Hn}\!\!\int_\dom\!\big(\Grad_x u(x+h)\cdot\hat{h}\big)\cdot\big(\Grad_x u(x)\cdot\hat{h}\big) \,dx \,d\mu_t(u)\omega(|h|) \,dh \,dt,
	\end{align*}
	so equation~\eqref{eq:khmhatu} becomes (after switching to polar coordinates)
	\begin{equation}\label{eq:lala}
	\begin{split}
	&\quad-\int_0^\infty\int_{\Hn}\fint_{\mathbb{S}^2}\int_{\dom}\! u(x)\cdot n\, u(x+r n)\cdot n \, \,dx\,dS(n)  \,d\mu_\tau(u)r^2\omega(r) \,dr\\
	&\quad+\int_0^\infty\int_{\Hn}\fint_{\mathbb{S}^2}\int_{\dom}\! u(x)\cdot n \,u(x+r n) \cdot n \,dx \,dS(n) \,d\mu_0(u)r^2\omega(r) \,dr\\
	&\quad-\frac12\int_0^\infty\int_0^\tau\int_{\Hn}\!\!\fint_{\mathbb{S}^2}\int_{\dom}\big(\big(u(x+r n)-u(x)\big)\cdot n\big)^3 \,dx\,dS(n) \,d\mu_t(u)\,dt\left(r^2\omega'(r)-2r\omega(r)\right) \,dr \\
	&\quad+\int_0^\infty\int_0^\tau\int_{\Hn}\!\!\fint_{\mathbb{S}^2}\int_{\dom}|u(x)-u(x+r n)|^2\big(u(x)-u(x+r n)\big)\cdot n  \,dx \,dS(n) \,d\mu_t(u)\,dt r\omega(r) \,dr\\
	&=2\vis\int_0^\infty\int_0^\tau\int_{\Hn}\!\!\fint_{\mathbb{S}^2}\int_{\dom}\!\big(\Grad_x u(x+r n)\cdot n\big)\cdot\big(\Grad_x u(x)\cdot n\big) \,dx \,dS(n) \,d\mu_t(u)\,dt r^2\omega(r) \,dr
	\end{split}
	\end{equation}
Denote
\begin{align*}
\widetilde{m}_2(\tau,r)&\coloneqq\int_{\Hn}\fint_{\mathbb{S}^2}\int_{\dom}\! u(x)\cdot n\, u(x+r n)\cdot n \, \,dx \,dS(n)  \,d\mu_\tau(u)\\
\widetilde{v}(\tau,r)&\coloneqq \int_0^\tau\int_{\Hn}\!\!\fint_{\mathbb{S}^2}\int_{\dom}\!\big(\Grad_x u(x+r n)\cdot n\big)\cdot\big(\Grad_x u(x)\cdot n\big) \,dx \,dS(n) \,d\mu_t(u)\,dt.
\end{align*}
Writing $\omega'(r)-2r^{-1}\omega(r) = r^2 \frac{d}{dr}\big(\omega(r)r^{-2}\big)$, \eqref{eq:lala} becomes
\begin{align*}
&\quad\frac12\int_0^\infty r^4 S_\|^3(\tau,r)\partial_r\big(r^{-2}\omega(r)\big) \,dr +\int_0^\infty S_0^3(\tau,r) r\omega(r) \,dr\\
&=-\int_0^\infty\widetilde{m}_2(\tau,r)r^2\omega(r) \,dr
+\int_0^\infty\widetilde{m}_2(0,r)r^2\omega(r) \,dr
-2\vis\int_0^\infty\widetilde{v}(\tau,r)r^2\omega(r) \,dr.
\end{align*}
Again, we note that due to the estimates from the energy inequality~\eqref{eq:energystatstrong}, $S^3_\|$, $\widetilde{m}$ and $\widetilde{v}$ are continuous and bounded quantities in $\tau$ and $r$. And so we can consider this ODE in the sense of distributions as an ODE in the strong sense,
\begin{equation*}
\partial_r \big(r^4S_\|^3(\tau,r)\big) = 2r^4\left(\frac{S_0^3(\tau,r)}{r} +\widetilde{m}_2(\tau,r)-\widetilde{m}_2(0,r)+2\vis\widetilde{v}(\tau,r)\right),
\end{equation*}
or
\begin{equation}
S_\|^3(\tau,r) = 2\int_0^r \frac{s^4}{r^4}\left(\frac{S_0^3(\tau,s)}{s} +\widetilde{m}_2(\tau,s)-\widetilde{m}_2(0,s)+2\vis\widetilde{v}(\tau,s)\right)\,ds.
\end{equation}
By the energy bound and Cauchy--Schwarz inequality, $\widetilde{m}_2$ and $\vis\widetilde{v}$ are uniformly bounded in $\vis$ for all $s,\tau>0$. Moreover from Lemma~\ref{lem:S30}, we have that $s^{-1}S_0^3(s,\tau)$ is uniformly bounded in $\vis$ by $2 E_0$. Hence,
\begin{equation}
\label{eq:Sparallelbound}
\left|\frac{S_\|^3(\tau,r)}{r}\right|\leq 2 E_0
\end{equation}
for some $C>0$ independent of $\epsilon$.
\end{proof}
\begin{remark}
	Combining~\eqref{eq:S_0bound} and~\eqref{eq:Sparallelbound}, we also obtain a uniform bound on $r^{-1} S^3_\perp(\tau,r)$, where $S^3_\perp$ is the transversal structure function
	\begin{equation}
	S^3_\perp(\tau,r)\coloneqq\int_0^\tau\int_{\Hn}\!\!\fint_{\mathbb{S}^2}\int_{\dom}|\delta^\perp_{r n}u(x)|^2\delta_{r n} u(x)\cdot n  \,dx \,dS(n) \,d\mu_t(u)\,dt = S^3_0(\tau,r)-S^3_{\|}(\tau,r),
	\end{equation}
	where
	\begin{equation}
\delta_h u(x) =u(x+h)-u(x),\quad  	\delta^\perp_{h}u = (\mathbf{I}-\hat{h}\otimes\hat{h})\delta_h u, \quad \hat{h}\coloneqq \frac{h}{|h|}.
	\end{equation}
	None of these quantities has a sign and therefore the previously derived bounds do not imply compactness without further assumptions.
\end{remark}

\subsection{Scaling assumption}\label{sec:scaling}
In order to pass to the limit $\vis\to 0$, we need to an additional assumption about the behavior of structure functions. Specifically, we need
\begin{assumption}[Weak statistical scaling]\label{ass:scaling}
	For any $\vis>0$, let $\mu^\vis$ be a Foia\c{s}--Prodi statistical solution of the incompressible Navier--Stokes equations. We assume that for $r\ll 1$, the second and third order longitudinal structure functions \eqref{eq:Spparallel} are related by
	\begin{equation}
	\left|S^2_{\|}(\tau,r)\right|\leq C \left|S^3_{\|}(\tau,r)\right|^{\alpha},
	\end{equation}
	where $C$ is a constant independent of $\vis$ and $\alpha>0$.
\end{assumption}

\begin{remark}[Weak statistical scaling]\label{rem:scaling}
Assumption \ref{ass:scaling} is inspired by the following stronger scaling assumption often encountered in turbulence theory: For any $p,q$ with $q\geq p\geq 1$, the $p$-th and $q$-th order longitudinal structure functions \eqref{eq:Spparallel} are related by
	\begin{equation}\label{eq:scaling2}
	\left|S^p_{\|}(\tau,r)\right|\leq C \left|S^q_{\|}(\tau,r)\right|^{\frac{\lambda(p)}{\lambda(q)}},
	\end{equation}
	where $C$ is a constant independent of $\vis$ and $\lambda(p)>0$ for $p\leq p_0$ where $3\leq p_0\in \R\cup\{\infty\}$. In Kolmogorov's 1941 (``K41'') theory~\cite{K41a,K41b,K41c}, $\lambda(p)=p$. However, this cannot be confirmed with physical experiments~\cite{Anselmet1984,Saw2018}. Various physicists therefore suggested intermittency corrections to account for the deviation from Kolmogorov's original theory, among others, Kolmogorov himself in 1962~\cite{Kolmogorov1962} in his refined theory of turbulence, Frisch et al. the $\beta$-model~\cite{Frisch1978}, as well as Novikov and Stewart~\cite{Novikov1964}. Assumption~\eqref{eq:scaling2} can also accommodate the frequently used model by She and Leveque~\cite{SheLeveque1994} who suggested
	\begin{equation}
	\lambda(p)= \frac{p}{9}+2\left(1-\left(\frac23\right)^{p/3}\right).
	\end{equation}
\end{remark}

\begin{remark}
	Combining the bound on the third order structure function in Lemma~\ref{lem:boundS3parallel} with Assumption~\ref{ass:scaling}, we obtain
	\begin{equation*}
	\left|S^2_{\|}(\tau,r)\right|\leq C r^{\alpha}.
	\end{equation*}
\end{remark}

We will combine Assumption \ref{ass:scaling} with the following lemma, which is Lemma 1 by Drivas~\cite{Drivas2021}, translated to the setting of statistical solutions. The proof is given in Appendix~\ref{sec:appendixC}:
\begin{lemma}[Weak anisotropy]\label{lem:Drivas}
	Let $\mu_t$ be a statistical solution of the Navier--Stokes equation. Then $\mu$ satisfies
	\begin{equation}\label{eq:WA}
	\begin{split}
	&3\int_0^T\int_{\dom}\int_{\hdiv}\fint_{\partial B_r(0)} (\delta_{rn} u\cdot n)^2dS(n)\,dx \,d\mu_t(u)\,dt\\
	&=\int_0^T\int_{\dom}\int_{\hdiv}\fint_{B_r(0)} |\delta_\ell u(x)|^2 d\ell \,dx \,d\mu_t(u)\,dt .
	\end{split}
	\end{equation}
\end{lemma}

Under Assumption \ref{ass:scaling}, we obtain
\begin{equation}
\int_0^T\int_{\dom}\int_{\hdiv}\fint_{B_r(0)} |\delta_\ell u(x)|^2 d\ell \,dx \,d\mu_t(u)\,dt\leq C r^\alpha.
\end{equation}
Using the equivalence theorem~\ref{thm:equivalencemeasures}, we can write this as
\begin{equation}\label{eq:DCvis}
\int_0^T\int_{\dom}\fint_{B_r(0)} \int_{U^2}|\xi_1-\xi_2|^2 \,d\nu_{x,x+y}^2(\xi) \,dy \,dx \,dt\leq C r^\alpha.
\end{equation}
and since this is uniform with respect to the viscosity coefficient $\vis$ (by the weak scaling assumption), it implies uniform diagonal continuity of the sequence $\{\cm\}_{\vis>0}$.
\subsection{Passage to the limit $\vis\to 0$}
Now we are in a position to prove our main result. We will keep track of the superscript $\vis$ again in order to distinguish between the approximating sequence $\{\cm^\vis\}_{\vis>0}$ and the limiting measure $\cm$ for $\vis=0$.
\begin{theorem}
	\label{thm:conv}
	Let $\{\cm^\vis\}_{\vis>0}$ be a sequence of (either Foia\c{s}--Temam or Friedman--Keller) statistical solutions to the Navier--Stokes equations with initial data $\mu_0$ with bounded support (cf.~\eqref{eq:initboundedsupp}). Assume that $\cm^\vis$ all satisfy Assumption~\ref{ass:scaling}. Then, as $\vis\to 0$, $\cm^\vis$ converges (along a subsequence) to a correlation measure $\cm$ on $L^2$ with bounded support (cf.~\eqref{eq:bdmoments})
	\begin{equation}
\int_{\dom^k}\int_{U^{k}} |\xi_1|^2\dots|\xi_k|^2 \,d\nu_{t,x}^k(\xi) \,dx\leq R^k<\infty,
\end{equation}
for some $0<R<\infty$, any $k\in \N$ and that satisfies the ``inviscid Friedman--Keller system'':
		\begin{multline}
	\label{eq:FKsysinviscid}
	\int_0^T\!\!\int_{\dom^k}\!\!\int_{U^k}\!\!(\xi_1\otimes \cdots\otimes \xi_k):\frac{\partial \varphi}{\partial t}(t,x)\,d\nu_{t,x}^k(\xi) \,dx \,dt+\int_{\dom^k}\!\!\int_{U^k}\!\!(\xi_1\otimes \cdots\otimes \xi_k):\varphi(0,x)\,d\nu_{0,x}^k(\xi) \,dx \\
	+\sum_{i=1}^k\int_0^T\!\!\int_{\dom^k}\!\!\int_{U^k}(\xi_1\otimes\cdots\otimes(\xi_i\otimes\xi_i)\otimes\cdots\xi_k):\Grad_{x_i}\varphi(t,x)\,d\nu_{t,x}^k(\xi) \,dx \,dt
	=0
	\end{multline}
	for all $k\in\N$, for all $\varphi\in C^2_c([0,T]\times \dom^k;U^k)$ with $\Div_{x_i}\varphi(x)=0$, a.e. $x\in \dom^k$ for all $i=1,\dots, k$
	and (corresponding to the divergence constraint)
	\begin{equation}
	\label{eq:divfreeconstlimit}
	\int_{\dom^{k}}\int_{U^{k}} \xi_1\otimes\dots\otimes\xi_{\ell}\otimes\alpha_{\ell+1}(\xi_{\ell+1})\otimes\dots \otimes \alpha_k(\xi_k) \,d\nu_{t,x}^k(\xi)\cdot\Grad_{x_1,\dots, x_\ell}\psi(x) \,dx = 0,
	\end{equation}
	where $\Grad_{x_1,\dots,x_\ell}=(\Grad_{x_{1}},\dots,\Grad_{x_\ell})^\top$, $1\leq \ell\leq k\in\N$, for all $\psi\in H^1(\dom^k;U^{k-\ell})$, $\alpha_j\in C(U;U)$, $\alpha_j(v)\leq C(1+|v|^2)$ and $j=1,\dots, k$.
\end{theorem}
\begin{proof}
	From the condition on the initially bounded support~\eqref{eq:initboundedsupp} and the energy inequality~\eqref{eq:energycm}, we obtain that the sequence $\cm^\vis$ satisfies~\eqref{eq:timedepuniformlpbound} for $p=2$ uniformly in $\vis>0$. The reasoning of Subsection~\ref{sec:KHM},~\ref{sec:ODEs} and~\ref{sec:scaling} resulting in~\eqref{eq:DCvis} imply that $\cm^\vis$ is uniformly diagonal continuous as in~\eqref{eq:timedepuniformdc}. Hence, using Theorem~\ref{thm:timedepcmcompactness}, we obtain, up to subsequence, the existence of a limiting correlation measure $\cm\in \Corrmeas^{2}([0,T),D;U)$. So it remains to check whether $\cm$ satisfies the equations~\eqref{eq:FKsysinviscid} and~\eqref{eq:divfreeconstlimit}. We note that the functions
	\begin{equation}
	\begin{split}
	g_1(t,x,\xi)&\coloneqq(\xi_1\otimes \cdots\otimes \xi_k):\frac{\partial \varphi}{\partial t}(t,x),\\
	g_2(t,x,\xi)&\coloneqq \sum_{i=1}^k (\xi_1\otimes \cdots\otimes \xi_k):\Delta_{x_i}\varphi(t,x),\\
g_3(t,x,\xi)&\coloneqq \sum_{i=1}^k (\xi_1\otimes\cdots\otimes(\xi_i\otimes\xi_i)\otimes\cdots\xi_k):\Grad_{x_i}\varphi(t,x)
	\end{split}
	\end{equation}
	for $\varphi\in C^2_c([0,T]\times \dom^k;U^k)$, $1\leq \ell\leq k\in \N$, are all functions in $\Caratheodory^{k,p}_1([0,T],D;U)$.
Hence, we can pass to the limit in all the terms in the Friedman--Keller system~\eqref{eq:FKsys}. The term that is multiplied by $\vis$ vanishes because it is a uniformly bounded in $\vis>0$ quantity that is multiplied by $\vis$. For the divergence constraint \eqref{eq:divfreeconstlimit}, we note that the function
\[
g_4(t,x,\xi)\coloneqq \theta(t)\xi_1\otimes\dots\otimes\xi_{\ell}\otimes\alpha_{\ell+1}(\xi_{\ell+1})\otimes\dots \otimes \alpha_k(\xi_k)\cdot\Grad_{x_1,\dots, x_\ell}\psi(x),
\]
lies in $\Caratheodory^{k,p}_1([0,T],D;U)$ for any $\theta\in C^\infty_c((0,T))$, and $\psi\in H^1(\dom^k;U^{k-\ell})$, and $\alpha_j\in C(U;U)$ with $|\alpha_j(v)|\leq C(1+|v|^2)$. Passing $\eps\to0$ in $\ip{\nu^{\eps,k}}{g_4}$ and using that $\nu^{\eps,k}$ satisfy the divergence constraint~\eqref{eq:divfreeconst},  we can conclude, by the arbitrariness of $\theta$, that~\eqref{eq:divfreeconstlimit} holds for a.e.\ $t\in [0,T]$.
\end{proof}
\begin{remark}
	By the equivalence theorem~\ref{thm:equivalencemeasures}, we know that the limiting correlation measure $\cm$ corresponds to a parametrized measure $\mu=(\mu_t)_{t>0}:[0,T)\to \mathcal{P}(\hdiv)$ that satisfies
	\begin{equation}
	\label{eq:statsolFTEuler}
	\int_{L^2(\dom)} \Phi(u) \,d\mu_t(u)=\int_{L^2(\dom)}\Phi(u) \,d\mu_0(u)+\int_0^t\int_{L^2(\dom)}\int_{\dom} (u(x)\otimes u(x)):\Grad_x \Phi'(u)(x) \,dx \,d\mu_s(u) \,ds
	\end{equation}
	for all cylindrical test functions $\Phi\in \cyl^0$ that satisfy $g_j\in C^2(\dom)$,
and the energy inequality
	\begin{equation}
	\label{eq:energystatEuler}
	\int_{L^2(\dom)} \|u\|_{L^2(\dom)}^2 \,d\mu_t(u)
	\leq 
	\int_{L^2(\dom)} \|u\|_{L^2(\dom)}^2 \,d\mu_0(u),\quad \text{for all } t\in [0,T].
	\end{equation}
	The proof of this fact follows along the lines of the proof of Theorem~\ref{thm:FKsys2FT_s} while ignoring the terms involving $\vis$ and not attempting to recover $B(u)$ as it may be unbounded.

\end{remark}

\section{Discussion}
It is well-known that many incompressible fluid flows of interest are characterized by very-high Reynolds number. Hence, a precise characterization of the vanishing viscosity ($\vis \to 0$) limit of the Navier-Stokes equations \eqref{eq:ns} is of great interest. Formally, one would expect that the vanishing viscosity limit of Navier-Stokes equations is related to the incompressible Euler equations. However as mentioned in the introduction, rigorous results in this direction are only available in two space dimensions, even in the case of periodic boundary conditions. The key aim of this article was to investigate the vanishing viscosity limit of the Navier-Stokes equations, including in three space dimensions. 

It is well known that fluid flows at high Reynolds numbers are characterized by \emph{turbulence}, loosely speaking, marked by the presence of energy containing eddies at smaller and smaller scales. This phenomenon is clearly linked to the lack of compactness in the Leray-Hopf Navier-Stokes solutions as well as their possible instabilities/non-uniqueness. 

Hence, one needs to make further assumptions on the Leray-Hopf solutions that can yield additional information and facilitate passage to the limit. One avenue for making such assumptions, which are realistic and possibly observed in experiments, comes from physical theories of turbulence. In particular, Kolmogorov's well-known K41 theory is based on several verifiable assumptions on the underlying fluid flow and results in a precise characterization of quantities such as structure functions and \emph{energy spectra}. 

In \cite{glimm1-2}, Chen and Glimm relate the K41 energy spectra to compactness results on the Leray-Hopf solutions, in appropriate Sobolev and H\"older spaces. Consequently, under the assumption of the K41 energy spectrum, the authors prove that the underlying Leray-Hopf solutions converge to \emph{weak solutions} of the incompressible Euler equations as $\vis \to 0$. However, Kolmogorov's derivation of the decay of energy spectra is based on a \emph{probabilistic charectization} of the underlying fluid flow. In particular, assumptions such as (statistical) homogeneity, isotropy and scaling, which form the foundation of Kolmogorov's theory, are too stringent if imposed at the deterministic level, as done in \cite{glimm1-2}. Moreover, it is now well-established that the strong scaling assumptions of Kolmogorov might not hold in real fluid flows and \emph{intermittent} corrections are necessary. Hence, the applicability of the assumptions and results of \cite{glimm1-2} can be questioned from this perspective.   

Nevertheless, the connection with Kolmogorov's theories of turbulence and their variants forms the basis of our work. We start with the realization that a probabilistic description of the solutions of Navier-Stokes equations is necessary to relate physical theories of turbulence to rigorous mathematical statements. To this end, we focus on \emph{statistical solutions} of Navier-Stokes equations. Two possible frameworks of such statistical solutions are available, namely the \emph{Foia\c{s}-Prodi} statistical solutions (see Definition \ref{def:foiastemamstatsol}) and the \emph{Friedman-Keller} statistical solutions (see Definition \ref{def:fksys}), which is based on the concept of correlation measures of \cite{FLM17}. We prove that both these solution concepts are equivalent as long as a statistical version of the energy inequality holds. This also allows us to prove the existence of Friedmann-Keller statistical solutions of the incompressible Navier-Stokes equations. 

Then, we investigated the vanishing viscosity limit of the statistical solutions of the incompressible Navier-Stokes equations. To this end, we derived a suitable statistical version of the well-known K\'arm\'an-Howarth-Monin relation and used it to prove precise rates for the asymptotic decay of a averaged third-order structure function in Lemma \ref{lem:S30}. However, these estimates do not suffice to pass to the $\vis \to 0$ limit. To this end, we assumed a \emph{weak statistical scaling} of the Navier-Stokes statistical solutions (see Assumption \ref{ass:scaling}). This assumption is a weaker version of Kolmogorov's scaling assumptions in his K41 theory. Moreover, it is consistent with and incorporates different variants of scaling that are proposed in the physics literature to explain intermittent corrections to Kolmogorov's theory. Under this assumption, we proved a weak anisotropy result and invoked compactness results of \cite{systemspaper} to rigorously prove that the statistical solutions of the Navier-Stokes equations converge, in a suitable sense, to a statistical solution of the incompressible Euler equations. Thus, we were able to characterize the vanishing viscosity limit of the Navier-Stokes equations in a relevant regime.  

At this juncture, it is essential to point that that no assumption, other than weak statistical scaling, is made in our results and all other estimates are derived rigorously. This should be contrasted with the results of \cite{glimm1-2} where the authors directly assume a decay of the energy spectrum for the weak solutions of the Navier-Stokes equations. It is currently unclear if one can relax the weak statistical scaling assumption or even if it holds for all incompressible fluid flows. Experimental evidence strongly supports that this assumption is verified in practice, see e.g.,~\cite{Anselmet1984,SheLeveque1994,Saw2018}. 

To the best of our knowledge, the only rigorous study of the vanishing viscosity limit of the (Foia\c{s}-Prodi) statistical solutions was carried out by Chae in \cite{Chae1991} where he proved that these statistical solutions converge to a measure-valued solution of the incompressible Euler equations. In contrast, we prove convergence to statistical solutions of the incompressible Euler equations and recall that statistical solutions are much more informative than measure-valued solutions as they also incorporate knowledge of all multi-point correlations. 

Finally, our characterization of the vanishing viscosity limit can be viewed in connection to recent results in \cite{MishraLanthalerPares2020} where the authors proved convergence of numerical spectral viscosity approximations to the statistical solutions of the Euler equations under very similar weak scaling assumptions.

\appendix

\section{Equivalence of different definitions of statistical solutions for the incompressible Navier--Stokes equations}\label{app:equivalence}
This appendix is devoted to the proof of Theorems~\ref{thm:FT2FKsys_s} and~\ref{thm:FKsys2FT_s}. For convenience, we restate the result:
\begin{theorem}[Foia\c{s}--Prodi statistical solutions are Friedman--Keller solutions]\label{thm:FT2FKsys}
		Let $\mu$ be a Foia\c{s}--Prodi statistical solution such that the initial condition $\mu_0$ has bounded support, $\text{supp}(\mu_0)\subset B\subset \hdiv$, $B=\{u\in \Lvec\, :\, \|u\|_{L^2(\dom)}\leq R\}$ for some $0<R\in\R$ large enough. Then $\mu$ is a Friedmann--Keller statistical solution (cf.~Definition~\ref{def:fksys}).
\end{theorem}
\begin{proof}
	It is shown in~\cite[Theorem 2, Section 3]{Foias1972} that Foia\c{s}--Prodi statistical solutions with initial measure $\mu_0$ having bounded support in $B_{\hdiv}(R)\coloneqq\{u\in\hdiv : \|u\|_{L^2(\dom)}\leq R\}$ have bounded support for all times, i.e., $\text{supp}(\mu_t)\subset B_{\hdiv}(R)$.
	Therefore, we can assume that $\mu_t$ has uniformly bounded support. This implies in particular that $\mu_t$ have bounded moments:
	\begin{equation}\label{eq:bdmoments2}
	\int_{\hdiv}\|u\|_{L^2(\dom)}^{2k} \,d\mu_t(u)\leq R^{2k}\quad \text{for a.e. } t\in[0,T].
	\end{equation}
	Moreover, by~\cite[Lemma 5, Section 3]{Foias1972}, we have that statistical solutions of Navier--Stokes satisfy
	\begin{multline}\label{eq:3.13i}
	\int_0^T\int_{\hdiv}\left[-\partial_t\Phi(t,u) + a(u,\partial_u\Phi(t,u))+ b(u,u,\partial_u\Phi(t,u))\right]\,d\mu_t(u)\,dt\\
	 =\int_{\hdiv}\Phi(0,u)\,d\mu_0(u)
	\end{multline}
	for any test function $\Phi(t,u)$ that is Fr\'{e}chet differentiable on $[0,T]\times \hdiv$ with $\Phi(t,\cdot)=0$ near $t=T$ and $|\partial_u\Phi(t,u)|\leq C$ and $|\partial_t\Phi(t,u)|\leq C_1+C_2\|u\|_{L^2(D)}$ for all $u$ and $t$ and some constants $C, C_1, C_2$ (equation ($3.13_{\mathrm{I}}$) and condition (3.8) in \cite{Foias1972}). Therefore, we can choose test functions
	\begin{equation*}
	\Phi(u)=q((u,\varphi_1),\dots,(u,\varphi_k))\theta(t)
	\end{equation*}
	for $q$ a polynomial on $\R^k$ and $\varphi_j\in \Vd\cap C^2_c(D,U)$ and $\theta\in C^1_c([0,T))$, $j=1,\dots, k$ in \eqref{eq:3.13i}, so that we get
	\begin{multline*}
	\int_0^T\!\!\!\!\int_{\hdiv}\!\!-\theta'(t)q((u,\varphi_1),\dots,(u,\varphi_k))d\mu_t(u)dt\\
	+ \int_0^T\!\!\!\!\int_{\hdiv}\!\!\Big[\theta(t)\sum_{i=1}^k\partial_i q((u,\varphi_1),\dots,(u,\varphi_k))\left[a(u,\varphi_i)+b(u,u,\varphi_i)\right]\Big]\,d\mu_t(u)\,dt\\
	=\int_{\hdiv}\!\!\!\!\!\theta(0)q((u,\varphi_1),\dots,(u,\varphi_k)) \,d\mu_0(u), 
		\end{multline*}
	Note that we can integrate by parts in the terms involving $a(u,\varphi_i)$ and $b(u,u,\varphi_i)$, $i=1,\dots, k$, so that all the derivatives are on the test functions $\varphi_i$, $i=1,\dots, k$ and $\theta$:
	\begin{align*}
	0=& \int_0^T\!\!\int_{\hdiv}\!\!\theta'(t)q((u,\varphi_1),\dots,(u,\varphi_k))\,d\mu_t(u)\,dt+\int_{\hdiv}\!\!\theta(0)q((u,\varphi_1),\dots,(u,\varphi_k)) \,d\mu_0(u)\\
	&\quad+ \sum_{i=1}^k\int_0^T\!\!\int_{\hdiv}\!\!\theta(t)\partial_i q((u,\varphi_1),\dots,(u,\varphi_k))\\
	&\hphantom{\quad +\sum_{i=1}^k\int_0^T\!\!\int_{\hdiv}\!\!}\times\int_{\dom}\left[\vis u(x_i)\Delta_{x_i}\varphi_i(x_i)+(u(x_i)\otimes u(x_i)):\Grad_{x_i}\varphi_i(x_i)\right] \,dx_i\,d\mu_t(u)\,dt
	\end{align*}
	Now take $q(s_1,\dots,s_k)=s_1\cdots s_k$, so that the last identity becomes (denote $\,dx=\,dx_1\dots \,dx_k$)
	\begin{multline*}
	0=\int_0^T\!\!\int_{\hdiv}\int_{\dom^k}\theta'(t)u(x_1)\cdot\varphi_1(x_1) \cdots u(x_k)\cdot\varphi_k(x_k) \,dx \,d\mu_t(u)\,dt\\
	+\sum_{i=1}^k\int_0^T\!\!\int_{\hdiv}\int_{\dom^k}\theta(t) \big(u(x_1)\cdot\varphi_1(x_1) \cdots \vis u(x_i)\cdot\Delta_{x_i}\varphi_i(x_i) \cdots u(x_k)\cdot\varphi_k(x_k)\big)   \,dx \,d\mu_t(u)\,dt\\
	+ \sum_{i=1}^k\int_0^T\!\!\int_{\hdiv}\int_{\dom^k}\theta(t) \big(u(x_1)\cdot\varphi_1(x_1)\cdots(u(x_i)\otimes u(x_i)):\Grad_{x_i}\varphi_i(x_i)\cdots u(x_k)\cdot\varphi_k(x_k)\big)   \,dx \,d\mu_t(u)\,dt\\
	+\int_{\hdiv}\int_{\dom^k}\theta(0)u(x_1)\cdot\varphi_1(x_1) \cdots u(x_k)\cdot\varphi_k(x_k) \,dx \,d\mu_0(u), 
	\end{multline*}
	which is, denoting $x=(x_1,\dots,x_k)$ and $\varphi(t,x)\coloneqq\theta(t) \varphi_1(x_1)\otimes\dots\otimes\varphi_k(x_k)$, equivalent to
	\begin{align*}
	0=&\ \int_0^T\!\!\int_{\hdiv}\int_{\dom^k}(u(x_1)\otimes\dots\otimes u(x_k)):\partial_t\varphi(t,x) \,dx \,d\mu_t(u)\,dt\\
	&+ \vis\sum_{i=1}^k\int_0^T\!\!\int_{\hdiv}\int_{\dom^k}(u(x_1)\otimes\dots\otimes u(x_k)):\Delta_{x_i}\varphi(t,x)  \,dx \,d\mu_t(u)\,dt\\
	&+ \sum_{i=1}^k\int_0^T\!\!\int_{\hdiv}\int_{\dom^k} (u(x_1)\otimes\dots\otimes(u(x_i)\otimes u(x_i))\otimes\dots\otimes u(x_k)):\Grad_{x_i}\varphi(t,x) \,dx \,d\mu_t(u)\,dt\\
&+\int_{\hdiv}\int_{\dom^k}(u(x_1)\otimes\dots\otimes u(x_k)):\varphi(0,x) \,dx \,d\mu_0(u), 
	\end{align*}
	Since we assume that the support of $\mu_0$ is bounded (and therefore also the support of $\mu_t$ for almost all $t\in [0,T]$), we can use a density argument to conclude that the above identity holds for all $\varphi\in L^2([0,T];(\Vd)^k)\cap C^2_c([0,T]\times D^k;U^k)$, that is, all $\varphi\in C^2_c([0,T]\times D^k;U^k)$ with $\Div_{x_i}\varphi=0$ for all $i=1,\dots,k$. Observe next that all the terms in the above identity have the form required to apply Theorem~\ref{thm:equivalencemeasures} (recall that $\mu_t$ has bounded support). Therefore, by  the same theorem, there exists a unique correlation measure $\cm = (\cm_t)_{0\leq t\leq T}=\big((\nu^1_{t},\nu^2_t,\dots)\big)_{0\leq t\leq T}$ corresponding to $\{\mu_t\}$ satisfying
	 \begin{align*}
	 0=&\ \int_0^T\!\!\int_{\dom^k}\int_{U^k}(\xi_1\otimes\dots\otimes \xi_k):\partial_t\varphi(t,x)  \,d\nu_{t,x}^k(\xi)\,dx\,dt\\
	 &+ \vis\sum_{i=1}^k\int_0^T\!\!\int_{\dom^k}\int_{U^k}(\xi_1\otimes\dots\otimes \xi_k):\Delta_{x_i}\varphi(t,x)  \,d\nu_{t,x}^k(\xi)\,dx\,dt\\
	 &+ \sum_{i=1}^k\int_0^T\!\!\int_{\dom^k}\int_{U^k} (\xi_1\otimes\dots\otimes(\xi_i\otimes \xi_i)\otimes\dots\otimes \xi_k):\Grad_{x_i}\varphi(t,x) \,d\nu_{t,x}^k(\xi)\,dx \,dt\\
	 &+\int_{\dom^k}\int_{U^k}(\xi_1\otimes\dots\otimes \xi_k):\varphi(0,x) \,d\nu_{0,x}^k(\xi)\,dx, 
	 	 \end{align*}
	which is \eqref{eq:FKsys}.
	By previous arguments, each measure $\mu_t$ has bounded support (and in particular bounded moments), so we can take any nonnegative, nondecreasing polynomial $\psi(s)=p_K(s)$, of degree $K$, for $K\in \N$ in the strengthened energy inequality \eqref{eq:energystatstrong}, so that we get
		\begin{multline}
	\label{eq:energystatmoments}
\int_{L^2(\dom)}p_K( \|u\|_{L^2(\dom)}^{2}) \,d\mu_t(u) + 2\vis  \int_0^t\int_{L^2(\dom)}p_K'(\|u\|_{L^2(\dom)}^{2})|u|_{H^1(\dom)}^2 \,d\mu_s(u)\,ds\\
	\leq 
	 \int_{L^2(\dom)} p_K(\|u\|_{L^2(\dom)}^{2}) \,d\mu_0(u)\qquad \text{for a.e. } t\in [0,T].
	\end{multline}
	Thanks to the bounded support of $\mu_t$, all the moment terms have the form required to apply Theorem~\ref{thm:equivalencemeasures}, moreover, with Lemma~\ref{lem:gradientrep2}, we can rewrite the gradient term such that we obtain for a.e.\ $t\in [0,T]$, the energy inequality \eqref{eq:energycm}. Indeed, we have with Theorem~\ref{thm:equivalencemeasures}
	\begin{align*}
	\int_{L^2(\dom)}p_K( \|u\|_{L^2(\dom)}^{2}) \,d\mu_t(u) &= \sum_{k=0}^Ka_k \int_{L^2(\dom)}\|u\|_{L^2(\dom)}^{2k} \,d\mu_t(u)\\
	&= \sum_{k=0}^Ka_k \int_{L^2(\dom)}\int_{\dom^k} |u(x_1)|^2\dots|u(x_k)|^2 dx \,d\mu_t(u)\\
	&= \sum_{k=0}^Ka_k  \int_{\dom^k}\int_{U^k} |\xi_1|^2\dots|\xi_k|^2 d\nu_{t,x}^k(\xi) dx.
	\end{align*}
	Similarly, with Lemma~\ref{lem:gradientrep2}, (via Lemma~\ref{lem:gradientrepresentation})
	\begin{align*}
	&\int_0^t\int_{L^2(\dom)}p_K'\big(\|u\|_{L^2(\dom)}^{2}\big)|u|_{H^1(\dom)}^2 \,d\mu_s(u)\,ds\\
	 &\qquad= \sum_{k=0}^Ka_k k\int_0^t\int_{L^2(\dom)}\|u\|_{L^2(\dom)}^{2(k-1)}|u|_{H^1(\dom)}^2 \,d\mu_s(u) \,ds\\
	&\qquad= \sum_{k=0}^Ka_k k\int_0^t\int_{L^2(\dom)}\int_{\dom^{k}}|u(x_1)|^2\dots |u(x_{k-1})|^2 |\Grad u(x_k)|^2 dx \,d\mu_s(u)\,ds\\
	&\qquad= \sum_{k=0}^Ka_k\sum_{i=1}^k \int_0^t\int_{L^2(\dom)}\int_{\dom^{k}}|u(x_1)|^2\dots |\Grad u(x_i)|^2\dots |u(x_{k})|^2  dx \,d\mu_s(u)\,ds\\
	&\qquad= \sum_{k=0}^Ka_k\sum_{i=1}^k \sum_{j=1}^d\lim_{h\to 0}\frac{1}{h^2}\int_0^t\int_{\dom^{k}}\int_{U^{k+1}}|\xi_1|^2\dots |\xi_i-\xi_{k+1}|^2\dots |\xi^k|^2 d\nu^{k+1}_{t,(x,x_i+h\mathbf{e}_j)}(\xi,\xi_{k+1}) dx\,ds
	\end{align*}
	and a similar computation for the term involving $\mu_0$ yields~\eqref{eq:energycm}.
The bounded support of $\mu_0$ implies, using the equivalence theorem~\cite[Theorem 2.7]{FLM17},
\begin{equation*}
\int_{\dom^k}\!\!\int_{U^{k}}|\xi_1|^2\dots|\xi_k|^2 \,d\nu_{0,x}^k(\xi)\,dx<\infty
\end{equation*}
for any $k\in\N$ and thus the boundedness of all terms in~\eqref{eq:energycm}.
It remains to show that each $\nu^k$ satisfies \eqref{eq:divfreeconst}. We let $\varphi \in C^1_c(\dom^k;U^k)$, $1\leq \ell\leq k$, $\alpha_j\in C(U;U)$, with $\alpha_j(v)\leq C(1+|v|^2)$, $j=\ell,\dots, k$ and compute using the equivalence theorem~\ref{thm:equivalencemeasures},
\begin{equation*}
\begin{split}
&	\int_{\dom^{k}}\int_{U^{k}} \xi_1\otimes\dots\otimes\xi_{\ell}\otimes\alpha_{\ell+1}(\xi_{\ell+1})\otimes\dots \otimes \alpha_k(\xi_k) \,d\nu_{t,x}^k(\xi)\cdot\Grad_{x_1,\dots, x_\ell}\varphi(x) \,dx\\
& \qquad= \int_{L^2(\dom)}\int_{\dom^k}\!\!\! \begin{aligned}[t]&u(x_1)\otimes \dots\otimes u(x_\ell)\otimes \alpha_{\ell+1}(u(x_{\ell+1}))\otimes\dots\otimes \alpha_k(u(x_k))\\
&\cdot \Grad_{x_1,\dots, x_\ell}\varphi(x) \,dx \,d\mu_t(u)\end{aligned}\\
&\qquad = 0,
\end{split}
\end{equation*}
since $\mu_t$ is supported on divergence free functions. This concludes the proof.
\end{proof}

Let us show the reverse direction now, that is, that any correlation measure that solves the Friedman--Keller system and satisfies in addition an energy inequality, is a statistical solution of the Navier--Stokes equations in the sense of Foia\c{s}--Prodi.
\begin{theorem}[Friedman--Keller solutions are Foia\c{s}--Prodi solutions]\label{thm:FKsys2FT}
	Let $\cm=(\cm_t)_{0\leq t\leq T}$ be a Friedman--Keller statistical solution of Navier--Stokes (cf.\ Definition \ref{def:fksys}) with bounded support, i.e.,
	\begin{equation}\label{eq:bdmoments3}
	\int_{\dom^k}\int_{U^{k}} |\xi_1|^2\dots|\xi_k|^2 \,d\nu_{t,x}^k(\xi) \,dx\leq R^k<\infty,
	\end{equation}
	for some $0<R<\infty$, every $k\in \N$, and almost every $t\in [0,T]$. Then $\cm$ corresponds to a probability measure $\mu=(\mu_t)_{0\leq t\leq T}$ on a bounded set of $\hdiv$ which is a Foia\c{s}--Prodi statistical solution of the Navier--Stokes equations (cf.~Definition~\ref{def:foiastemamstatsol}).
\end{theorem}

\begin{proof}
From the equivalence theorem~\ref{thm:equivalencemeasures}, we obtain that $(\cm_t)_{0\leq t\le T}$ corresponds to a family of measures $(\mu_t)_{0\leq t\leq T}\subset\mathcal{P}(L^2(\dom;U))$ with bounded support, that is $\text{supp}(\mu_t)\subset \{u\in L^2(\dom;U)\,:\, \|u\|_{L^2(\dom;U)}\leq R\}$ for almost every $t\in [0,T]$. Property {\bf (a)} of Definition \ref{def:foiastemamstatsol} -- the fact that \eqref{eq:measurabilitymu} is measurable for all $\phi\in C_b(L^2)$ -- follows from a monotone class argument, which we include here. By property (i) of Theorem \ref{thm:equivalencemeasures}, the function
\begin{equation*}
t\mapsto \int_{L^2(\dom;U)}L_f(u) \,d\mu_t(u) = \int_{L^2(\dom;U)}\int_D f(x,u(x)) \,dx\,d\mu_t(u)
\end{equation*}
is measurable for every $f\in  L^1(\dom^k,C_0(U^k))$. Let $\mathbf{M}$ be the collection of sets
\[
\mathbf{M} = \left\{E\in \Borel(L^2(D;U)) \text{ such that } t \mapsto \int_{L^2(D;U)} \ind_E(u)\,d\mu_t(u) \text{ is measurable} \right\}.
\]
By the monotone convergence theorem, $\mathbf{M}$ is a monotone class, that is, it is closed under (countable) unions of increasing sequences of sets and intersections of decreasing sequences of sets. By the same argument as in the proof of \cite[Proposition 2.12]{FLM17}, $\mathbf{M}$ contains the collection of cylinder sets $\mathrm{Cyl}(L^2) = \{\text{all cylinder sets on }L^2(D;U)\}$, that is, all sets of the form $E = \big\{u\in L^2\ :\ (\ip{\phi_1}{u},\dots,\ip{\phi_n}{u})\in F\big\}$ for some $n\in\N$, a Borel set $F\subset\R^n$ and $\phi_1,\dots,\phi_n\in L^2(D;U)$. Since $\mathrm{Cyl}(L^2)$ is an algebra which generates $\Borel(L^2)$ (cf.~e.g.~\cite[Appendix]{FLM17}), it follows from the monotone class lemma that $\mathbf{M} = \sigma(\mathrm{Cyl}(L^2)) = \Borel(L^2)$. Approximating an arbitrary $\phi\in C_b(L^2(D;U))$ by simple functions now gives the desired conclusion.

We claim that $\mu_t$ is supported on $\hdiv$. This follows from Lemma~\ref{lem:divconstraintlem}: Since $\mu$ has bounded support on $L^2(\dom)$ we may take $\phi(\xi)=|\xi|^2$ in that lemma, which is continuous and bounded on any compact subset of $\R$ and satisfies $\phi(0)=0$. Then by the lemma, for any $g\in H^1(\dom)$,
\begin{equation*}
\int_{L^2(\dom)}\left|\int_{\dom} u\cdot\Grad g \,dx\right|^2\,d\mu_t(u) =0.
\end{equation*}
Hence, by Chebychev's inequality,
\begin{equation*}
\mu_t\left(\left\{u\in L^2(\dom): \,\int_{\dom} u\cdot\Grad g \,dx\neq 0\right\}\right)=0.
\end{equation*}
Since $g\in H^1(\dom;U)$ was arbitrary and $H^1(\dom;U)$ is separable, this implies that $\mu_t$ is supported on $L^2$-functions that are weakly divergence free, which is exactly the space $\hdiv$ (see e.g.~\cite[Section 1, Chapter 1]{Temam2001}).

Next we claim that $\mu$ satisfies condition {\bf (c)} in Definition~\ref{def:foiastemamstatsol}. As $\cm$ is assumed to satisfy~\eqref{eq:energycm}, we can apply \cite[Theorem 2.7]{FLM17} combined with Lemma~\ref{lem:gradientrep3} to each of the terms and obtain that $\mu_t$ satisfies~\eqref{eq:energystatmoments} for any nonnegative and nondecreasing polynomial $p_K(s)=\sum_{k=0}^K a_k s^k$ on $[0,R]$. This implies in particular that $\mu_t$ is supported on $\Hvec$, for a.e.~$t\in[0,T)$, and~\eqref{eq:h1linf}. Now any differentiable nondecreasing function $\psi$ on a bounded interval can be approximated by nondecreasing polynomials~\cite{Shisha1965}, from which~\eqref{eq:energystatstrong} follows after passing to the limit in a suitable polynomial approximation.
Specifically, for a given $\psi$, let $\{p_{K_n}\}_{n\in\N}$ be a sequence of nonnegative, nondecreasing polynomials ($K_n\to\infty$) satisfying
\begin{equation*}
\norm{p_{K_n}-\psi}_{C^1([0,R])}\leq \frac{1}{n}.
\end{equation*}
Then, for $t\geq 0$, by the compact support property of $\mu$ on $B$,
\begin{equation*}
\left|\int_{L^2(\dom)}\psi(\norm{u}_{L^2(\dom)}^2)d\mu_t(u)-\int_{L^2(\dom)} p_{K_n}(\norm{u}^2_{L^2(\dom)})d\mu_t(u)\right|\leq \norm{p_{K_n}-\psi}_{C^1([0,R])}\leq \frac{1}{n},
\end{equation*}
and
\begin{align*}
&\left|\int_0^t\int_{L^2(\dom)}\psi'(\norm{u}_{L^2(\dom)}^2) |u|^2_{H^1(\dom)}d\mu_t(u)ds-\int_0^t\int_{L^2(\dom)} p_{K_n}'(\norm{u}^2_{L^2(\dom)})|u|^2_{H^1(\dom)} d\mu_t(u)ds \right|\\
& \qquad\leq \int_0^t\int_{L^2(\dom)}\left|\psi'(\norm{u}_{L^2(\dom)}^2)-p_{K_n}'(\norm{u}^2_{L^2(\dom)})\right| |u|^2_{H^1(\dom)}d\mu_t(u)ds\\
&\qquad\leq \norm{p_{K_n}-\psi}_{C^1([0,R])}\int_0^t\int_{L^2(\dom)}|u|^2_{H^1(\dom)}d\mu_t(u)ds\\
&\qquad\leq \frac{C}{n},
\end{align*}
where the last inequality follows from~\eqref{eq:h1linf}.

It remains to show that the correlation measures satisfy the evolution equation~\eqref{eq:statsolFT} for all $\Phi\in \cyl^0$. From this also part {\bf (d)} of the definition will follow (see Remark~\ref{rem:integrability}). Let $\Phi(u) = \phi((u,g_1),\dots,(u,g_k))$ be an arbitrary function in $\cyl^0$ with $g_j\in \Hvec$ and let $\widetilde{R}=\max_{1\leq j\leq k}(\norm{g_j}_{L^2})$. Since $\phi$ is continuously differentiable and bounded on $B\coloneqq[-R\widetilde{R}-1,R\widetilde{R}+1]^k$, we can approximate it arbitrarily well by polynomials thanks to the Weierstrass approximation theorem. Let
\begin{equation*}
p_n(\zeta)\coloneqq\sum_{|\bj|=0}^{N_n}\beta_{\bj}\zeta_1^{j_1}\dots\zeta_k^{j_k},\quad j_i\geq 0,\quad \bj=(j_1,\dots,j_k),\quad |\bj|\coloneqq j_1+\dots+j_k,
\end{equation*}
for some $N_n\in\N$ large enough, be approximations of $\phi$ that satisfy
\begin{equation*}
 \norm{\phi-p_n}_{C^1(B)}\leq \frac{1}{n},\quad n\in\N.
\end{equation*}
Let $\theta\in C_c^1((0,T))$ be an arbitrary compactly supported test function. Since by the equivalence theorem, for any $\bj=(j_1,\dots, j_k)$, $j_i\geq 0$, $|\bj|=j_1+\dots +j_k$,
\begin{align*}
&\int_0^T\theta'(t)\int_{\dom^{|\bj|}}\left(\int_{U^{|\bj|}}\xi_{1}\otimes\cdots\otimes \xi_{|\bj|} \,d\nu_{t,x}^{|\bj|}(\xi)\right) \\
&: \Big(g_1(x_1)\otimes\cdots \otimes g_1(x_{j_1})\otimes \cdots \otimes g_k(x_{|\bj|-j_k+1})\otimes \cdots\otimes g_k(x_{|\bj|})\Big)  \,dx \,dt\\
&= \int_0^T\theta'(t) \int_{L^2(\dom)}(u,g_1)^{j_1}\dots (u,g_k)^{j_k} \,d\mu_t(u) \,dt,
\end{align*}
and if $\widetilde{j}=j_1+\dots + j_{i-1}$, then
\begin{align*}
&\sum_{\ell=\widetilde{j}+1}^{\widetilde{j}+j_i}\int_0^T\!\!\!\theta(t)\!\int_{\dom^{|\bj|}}\left(\int_{U^{|\bj|}} \xi_{1}\otimes\cdots\otimes \xi_{|\bj|} \,d\nu_{t,x}^{|\bj|}(\xi)\right)\\
&:  \Big(g_1(x_1)\otimes\cdots \otimes g_i(x_{\widetilde{j}+1})\otimes \dots \otimes \Delta_{x_{\ell}}g_i(x_{\ell})\otimes\dots\otimes g_i(x_{\widetilde{j}+j_i})\otimes\dots  \otimes g_k(x_{|\bj|})\Big)  \,dx \,dt\\
=&\ j_i\int_0^T\!\!\theta(t)\int_{L^2(\dom;U)}(u,g_1)^{j_1}\cdots (u, g_i)^{j_i-1}(u,\Delta g_i)\cdots (u,g_k)^{j_k} \,d\mu_t(u)  \,dt\\
=&\ j_i\int_0^T\!\!\theta(t)\int_{L^2(\dom;U)}(u,g_1)^{j_1}\cdots (u,g_i)^{j_i-1}(Au,g_i)\cdots (u,g_k)^{j_k} \,d\mu_t(u)  \,dt,
\end{align*}
where the last equality follows because $g_i$ is divergence free because $\mu_t$ is supported on $\Hvec$-functions by the energy inequality~\eqref{eq:energystatstrong}. Furthermore, we have
\begin{align*}
&\sum_{\ell=\widetilde{j}+1}^{\widetilde{j}+j_i}\int_0^T\!\!\!\theta(t)\!\int_{\dom^{|\bj|}}\left(\int_{U^{|\bj|}}\xi_{1}\otimes\cdots\otimes\xi_{\widetilde{j}+1}\otimes\cdots\otimes(\xi_{\ell}\otimes \xi_{\ell})\otimes\cdots\otimes\xi_{\widetilde{j}+j_i}\otimes\cdots\otimes \xi_{|\bj|} \,d\nu_{t,x}^{|\bj|}(\xi)\right)\\
&: \Big( g_1(x_1)\otimes \cdots \otimes\Grad_{x_{\ell}}g_i(x_{\ell})\otimes \cdots \otimes g_k(x_{|\bj|})\Big) \,dx \,dt\\
=&\ j_i\int_0^T\!\!\theta(t) \int_{L^2(\dom)}(u,g_1)^{j_1}\cdots (u, g_i)^{j_i-1}(u\otimes u,\Grad g_i)\cdots (u,g_k)^{j_k} \,d\mu_t(u) \,dt\\
=&\ -j_i\int_0^T\!\!\theta(t) \int_{L^2(\dom)}(u,g_1)^{j_1}\cdots (u,g_i)^{j_i-1}(B(u),g_i)\cdots (u,g_k)^{j_k} \,d\mu_t(u) \,dt,
\end{align*}
again, the last equality following because $g_i$ is divergence free and $\mu_t$ is supported on $\Hvec$-functions for almost every $t\in[0,T]$.
We obtain, by combining the terms to the Friedman--Keller system, that if
\begin{equation*}
\Phi_n(u) \coloneqq p_n\big((u,g_1),\dots,(u,g_k)\big) 
\end{equation*}
then $\mu$ satisfies the equation
\begin{equation}\label{eq:approxstatsol}
\int_0^T\int_{L^2(\dom)}\theta'(s)\Phi_n(u) \,d\mu_s(u)\,ds + \int_0^T\theta(s)\int_{L^2(\dom)} (F(s,u),\partial_u\Phi_n(u)) \,d\mu_s(u) \,ds=0
\end{equation}
for every test function $\theta\in C_c^1((0,T))$.
Next, note that
\begin{equation}\label{eq:approx1}
\begin{split}
&\left|\int_{L^2(\dom)}\Phi_{n}(u)\,d\mu_t(u) - \int_{L^2(\dom)}\Phi(u) \,d\mu_t(u)\right|=\left|\int_{L^2(\dom)}\left(\Phi_{n}(u)-\Phi(u)\right) \,d\mu_t(u)\right|\\
&\quad\leq \int_{L^2(\dom)}\big|p_n\big((u,g_1),\dots,(u,g_k)\big)-\phi\big((u,g_1),\dots,(u,g_k)\big)\big| \,d\mu_t(u)\\
&\quad \leq \int_{L^2(\dom)}\sup_{\zeta\in B}\left|p_n(\zeta)-\phi(\zeta)\right| \,d\mu_t(u)\\
&\quad= \norm{\phi-p_n}_{C^0(B)} \leq \frac{1}{n},
\end{split}
\end{equation}
because $\mu_t$ has support on $\{u\in L^2(\dom)\,:\, \|u\|_{L^2(\dom)}\leq R\}$. Moreover,
\begin{equation}\label{eq:approx2}
\begin{split}
&\left|\int_0^T\theta(s)\int_{L^2(\dom)}\big(\partial_u\Phi_n(u),F(s,u)\big)\,d\mu_s(u)\,ds - \int_0^T\theta(s)\int_{L^2(\dom)}\big(\partial_u\Phi(u),F(s,u)\big) \,d\mu_s(u)\,ds\right|\\
&\quad\leq\|\theta\|_{L^\infty}\int_0^T\int_{L^2(\dom)}\big|\big(\partial_u(\Phi(u)-\Phi_n(u)),F(s,u)\big)\big| \,d\mu_s(u)\,ds\\
&\quad \leq\|\theta\|_{L^\infty}\sum_{j=1}^k\int_0^T\int_{L^2(\dom)}\big|\partial_j\phi\big((u,g_1),\dots,(u,g_k)\big)-\partial_jp_n\big((u,g_1),\dots,(u,g_k)\big)\big|\big|(g_j,F(s,u))\big| \,d\mu_s(u)\,ds\\
&\quad \leq \|\theta\|_{L^\infty}\sum_{j=1}^k\sup_{B}\left|\partial_j(\phi-p_n)\right|\int_0^T\int_{L^2(\dom)}\big|(g_j,F(s,u))\big| \,d\mu_s(u)\,ds\\
&\quad\leq \|\theta\|_{L^\infty}\norm{\phi-p_n}_{C^1(B)}\sum_{j=1}^k\int_0^T\int_{L^2(\dom)}\big|(g_j,F(s,u))\big| \,d\mu_s(u)\,ds\\
&\quad\leq
\frac{1}{2n}\|\theta\|_{L^\infty} \sum_{j=1}^k\left(\norm{g_j}_{H^1(\dom)}^2+\int_0^T\int_{L^2(\dom)}
\norm{\Grad u}_{L^2(\dom)}^2 \,d\mu_s(u)\,ds \right)
\end{split}
\end{equation}
where we used Young's inequality for the last inequality.
Thanks to \eqref{eq:approx1} and \eqref{eq:approx2}, we can pass to the limit $n\rightarrow \infty$ in \eqref{eq:approxstatsol} and obtain
\begin{equation}\label{eq:approxstatsol2}
\int_0^T\theta'(s)\int_{L^2(\dom)}\Phi(u) \,d\mu_s(u)\,ds + \int_0^T\theta(s)\int_{L^2(\dom)} (F(s,u),\partial_u\Phi(u)) \,d\mu_s(u) \,ds=0
\end{equation}
for every $\theta\in C_c^1((0,T))$. It follows that the distributional derivative of $\varphi(s)\coloneqq \int_{L^2(\dom)}\Phi(u) \,d\mu_s(u)$ is
\[
\frac{d}{ds}\varphi(s) = \int_{L^2(\dom)}(F(s,u),\partial_u\Phi(s,u))\,d\mu_s(u),
\]
and since the right-hand side lies in $L^1((0,T))$ we see that $\varphi$ is absolutely continuous. Consequently, we can (through a standard approximation procedure) insert $\theta=\ind_{(0,t)}$ in \eqref{eq:approxstatsol2} and conclude that \eqref{eq:statsolFT} is true.

\end{proof}

	\section{Auxiliary lemmata}\label{appendix:auxilary}
	\begin{lemma}[Representation of the gradient]
		\label{lem:gradientrepresentation}
	Let $(\mu_t)_{0\leq t\leq T}$ be a measure on $L^2(\dom)$ satisfying
	\begin{equation}\label{eq:bdl2gradient}
	\int_0^T\int_{L^2(\dom)}\norm{\Grad u}_{L^2(\dom)}^2 \,d\mu_t(u) \,dt\leq C<\infty.
	\end{equation}
	Define for $h\in\R$ the finite difference gradient $\Grad_h$ by
	\begin{equation*}
	\Grad_h f(x) = \left(	D^h_1 f(x),\dots, D^h_d f(x)\right), \quad D^h_j f(x) = \frac{f(x+ h \mathbf{e}_j)- f(x)}{h},
	\end{equation*}
	where $\mathbf{e}_j\in \R^d$ is the $j$th unit vector. Then we have for any $V\subset\subset \dom$ (or in the case that $\dom$ is the torus also for $V=\dom$)
	\begin{equation}
	\label{eq:gradientid}
		\int_0^T\int_{L^2(\dom)}\norm{\Grad u}_{L^2(V)}^2 \,d\mu_t(u) \,dt = \lim_{h\rightarrow 0} \int_0^T\int_{L^2(\dom)}\norm{\Grad_h u}_{L^2(V)}^2 \,d\mu_t(u) \,dt.
	\end{equation}
	\end{lemma}
	\begin{proof}
		Consider a compact $V\subset\subset \dom$  and a smooth $u\in C^\infty(\dom)$. Then, since we can write
		\begin{equation}
		D^h_j u(x) = \int_0^1 \partial_{x^j}u(x+t h \mathbf{e}_j) \,dt,
		\end{equation}
		and hence
		\begin{equation}
		|\Grad_h u(x)|\leq \sum_{j=1}^d\int_{0}^{1}|\partial_{x^j}u(x+t h \mathbf{e}_j)| \,dt\leq \sum_{j=1}^{d}\norm{\partial_{x^j}u}_{L^\infty(V)}\leq C\norm{\Grad u}_{L^\infty(V)},
		\end{equation}
		for all $x\in V$ and $\int_{V}C\norm{\Grad u}_{L^\infty(V)} \,dx<\infty$, we can use the dominated convergence theorem to obtain that
		\begin{equation*}
		\int_{V}|\Grad u|^2 \,dx =\lim_{h\rightarrow 0}\int_V |\Grad_h u|^2 \,dx.
		\end{equation*}
	Moreover, by~\cite[Theorem 3, Section 5.8.2]{Evans2010}, we have for any $u\in H^1(\dom)$,
		\begin{equation}\label{eq:evans582}
		\norm{\Grad_h u}_{L^2(V)}\leq C_V \norm{\Grad u}_{L^2(\dom)},
		\end{equation}
		for some constants $0\leq c_V,C_V<\infty$ and any $V\subset\subset\dom$.
	Since $C^\infty(\dom)$ is dense in $H^1(\dom)$,
	we can find a sequence of functions $(u_n)_{n\in\N}\subset C^\infty(\dom)$ such that
	\begin{equation*}
	\lim_{n\rightarrow \infty}\norm{u_n-u}_{H^1(V)} = 0,
	\end{equation*}
	for any $V\subset \dom$ bounded.
	So fix $V\subset\subset \dom$ and an arbitrary $\eps>0$ and pick $N_\eps\in\N$ large enough such that for all $n\geq N_\eps$,
	\begin{equation*}
	\norm{u-u_n}_{H^1(V)}\leq \frac{1}{2}\frac{1}{1+C_V}\eps,
	\end{equation*}
	where $C_V$ is the constant in \eqref{eq:evans582}.
	Then for $n\geq N_\eps$
	\begin{align*}
	\norm{\Grad u - \Grad_h u}_{L^2(V)}&\leq \norm{\Grad (u-u^n)}_{L^2(V)}+ \norm{\Grad u^n -\Grad_h u^n}_{L^2(V)}+ \norm{\Grad_h (u^n-u)}_{L^2(V)}\\
	&\leq \norm{\Grad (u-u^n)}_{L^2(V)}+ \norm{\Grad u^n -\Grad_h u^n}_{L^2(V)}+ C_V\norm{\Grad (u^n-u)}_{L^2(V)}\\
	&\leq \frac{\eps}{2}+\norm{\Grad u^n -\Grad_h u^n}_{L^2(V)}.
	\end{align*}
	Now pick $h_0$ with $|h_0|$ small enough that
	\begin{equation*}
	\norm{\Grad u^n -\Grad_h u^n}_{L^2(V)}\leq\frac{\eps}{2}
	\end{equation*}
	for all $h$ with $|h|\leq |h_0|$.
	Then
	\begin{equation*}
	\norm{\Grad u - \Grad_h u}_{L^2(V)}\leq \eps
	\end{equation*}
	for all $h$ with $|h|\leq |h_0|$. Since $\eps>0$ was arbitrary, we obtain
	\begin{equation}\label{eq:convgrad}
	\lim_{h\rightarrow 0} \int_V |\Grad_h u|^2 \,dx = \int_V |\Grad u|^2 \,dx.
	\end{equation}
	Combining~\eqref{eq:evans582} with \eqref{eq:bdl2gradient}, we can apply the dominated convergence theorem and use \eqref{eq:convgrad} to pass to the limit:
	\begin{equation*}
	\begin{split}
			\int_0^T\int_{L^2(\dom)}\norm{\Grad u}_{L^2(V)}^2 \,d\mu_t(u) \,dt& =  \int_0^T\int_{L^2(\dom)}\lim_{h\rightarrow 0}\norm{\Grad_h u}_{L^2(V)}^2 \,d\mu_t(u) \,dt\\
			&= \lim_{h\rightarrow 0} \int_0^T\int_{L^2(\dom)}\norm{\Grad_h u}_{L^2(V)}^2 \,d\mu_t(u) \,dt.
	\end{split}
\end{equation*}
\end{proof}
\begin{corollary}\label{cor:bla}
	Let $(\mu_t)_{0\leq t\leq T}$ be a measure on $L^q(\dom)$. If $\mu_t$ satisfies
	\begin{equation*}
	\int_0^T\int_{L^q(\dom)}\norm{\Grad u}_{L^q(\dom)}^q \,d\mu_t(u) \,dt\leq C<\infty,
	\end{equation*}
	then for any $V\subset\subset \dom$ (or if $\dom$ is the torus also for $V=\dom$)
	\begin{equation*}
	\int_0^T\int_{L^q(\dom)}\norm{\Grad u}_{L^q(V)}^q \,d\mu_t(u) \,dt =\lim_{h\rightarrow 0}\int_0^T\int_{L^q(\dom)}\norm{\Grad_h u}_{L^q(V)}^q \,d\mu_t(u) \,dt.
	\end{equation*}

\end{corollary}
\begin{lemma}
	[More representation of gradients]\label{lem:gradientrep2}
	Let $(\mu_t)_{0\leq t\leq T}$ be a measure on $L^q(\dom)$, where $\dom$ is the $d$-dimensional torus satisfying
	\begin{equation}\label{eq:aprioribd}
	\int_0^T\int_{L^q(\dom)} \norm{\Grad u}_{L^q(\dom)}^q \,d\mu_t(u)\,dt\leq C<\infty.
	\end{equation}
	Then we can represent the integral of the gradient in terms of correlation measures as
	\begin{equation}\label{eq:corgradient}
	\int_0^T\int_{L^q(\dom)} \norm{\Grad u}_{L^q(\dom)}^q \,d\mu_t(u)\,dt = \lim_{h\rightarrow 0}\frac{1}{h^q}\int_0^T \int_{\dom}\sum_{j=1}^d\int_{U^2}|\xi_{1}-\xi_2|^q \,d\nu^2_{x,x+h\mathbf{e}_j}(\xi) \,dx \,dt.
	\end{equation}
		\end{lemma}
\begin{proof}
	From Lemma \ref{lem:gradientrepresentation}, we know that we can represent the left hand side of \eqref{eq:corgradient} as the limit as $h$ goes to zero of
	\begin{equation*}
	\frac{1}{h^q}\int_0^T\int_{L^q(\dom)}\int_{\dom}|u(x+h)-u(x)|^q \,dx \,d\mu_t(u) \,dt.
		\end{equation*}

Note that for arbitrary nonnegative functions $\varphi\in L^1(\dom)$ with $\varphi_\eps (x)\coloneqq\eps^{-d}\varphi(x/\eps)$, we have that
\begin{equation*}
\int_{\dom}\int_{\dom}|f(x)-f(x-y)|^q\varphi_\eps(y) \,dy \,dx\stackrel{\eps\to 0}{\longrightarrow} 0,
\end{equation*}
for any $f\in L^q(\dom)$, since shifts are continuous in $L^p$ for $L^p$-functions.
In particular, choosing $\varphi(x)=\frac{1}{|B_1(0)|}\mathbbm{1}_{B_1(0)}(x)$ this yields
\begin{equation}\label{eq:lpcont}
\int_{\dom}\intavg_{B_\eps(x)} |f(x)-f(y)|^q \,dy \,dx\stackrel{\eps\to 0}{\longrightarrow} 0,
\end{equation}
and
\begin{equation*}
\int_{\dom}\intavg_{B_\eps(x+h\mathbf{e}_j)} |f(x)-f(y)|^q \,dy \,dx\stackrel{\eps\to 0}{\longrightarrow} \int_{\dom} |f(x)-f(x+h\mathbf{e}_j)|^q  \,dx,
\end{equation*}
for any fixed $h>0$ and $\mathbf{e}_j$ is the $j$th unit vector.
We have for any $\eps>0$,
\begin{align}
&\Biggl|\frac{1}{h}\left(\int_0^T\!\!\!\int_{L^q(\dom)}\int_{\dom}\!\!\!|u(x)-u(x+h\mathbf{e}_j)|^q  \,dx\,d\mu_t(u) \,dt\right)^{1/q}\\
&\quad\qquad -\frac{1}{h}\left(\int_0^T\!\!\!\int_{L^q(\dom)}\int_{\dom}\intavg_{B_\eps(x+h\mathbf{e}_j)} |u(x)-u(y)|^q \,dy \,dx\,d\mu_t(u) \,dt\right)^{1/q}\Biggr|\\
&\quad\leq \left(\frac{1}{h^q}\int_0^T\!\!\!\int_{L^q(\dom)}\int_{\dom}\intavg_{B_\eps(x+h\mathbf{e}_j)} |u(x+h\mathbf{e}_j)-u(y)|^q \,dy \,dx\,d\mu_t(u) \,dt\right)^{1/q}\\
&\quad = \left(\frac{1}{h^q}\int_0^T\!\!\!\int_{L^q(\dom)}\int_{\dom}\intavg_{B_\eps(x)} |u(x)-u(y)|^q \,dy \,dx\,d\mu_t(u) \,dt\right)^{1/q} \stackrel{\eps\to 0}{\longrightarrow} 0,
\end{align}
by~\eqref{eq:lpcont} and Lebesgue dominated convergence theorem.
Hence
\begin{multline}
\frac{1}{h^q}\int_0^T\!\!\!\int_{L^q(\dom)}\int_{\dom}\!\!\!|u(x)-u(x+h\mathbf{e}_j)|^q  \,dx\,d\mu_t(u) \,dt \\
=\frac{1}{h^q}\lim_{\eps\to 0}\int_0^T\!\!\!\int_{L^q(\dom)}\int_{\dom}\intavg_{B_\eps(x+h\mathbf{e}_j)} |u(x)-u(y)|^q \,dy \,dx\,d\mu_t(u) \,dt.
\end{multline}
Now we can apply~\cite[Theorem 2.7]{FLM17} to rewrite the last expression in terms of the unique correlation measure $\cm_t$ to which $\mu_t$ corresponds:
\begin{multline*}
\frac{1}{h^q}\lim_{\eps\rightarrow 0}\int_0^T\!\!\!\int_{L^q(\dom)}\int_{\dom}\intavg_{B_\eps(x+h\mathbf{e}_j)} |u(x)-u(y)|^q \,dy \,dx\,d\mu_t(u) \,dt \\
= \frac{1}{h^q}\lim_{\eps\rightarrow 0}\int_0^T\!\!\!\int_{\dom}\intavg_{B_\eps(x+h\mathbf{e}_j)}\int_{U^{2}} |\xi_{1}-\xi_{2}|^q \,d\nu_{x,y}^2(\xi) \,dy \,dx \,dt
\end{multline*}
Because of the equivalence theorem~\cite[Theorem 2.7]{FLM17} and assumption \eqref{eq:aprioribd} combined with Corollary~\ref{cor:bla}, we have
\begin{equation*}
\lim_{\eps\rightarrow 0}\frac{1}{h^q}\int_0^T\int_{\dom}\intavg_{B_\eps(x+h\mathbf{e}_j)}\int_{U^{2}}|\xi_{1}-\xi_{2}|^q \,d\nu_{x,y}^2(\xi) \,dy \,dx \,dt<\infty,
\end{equation*}
uniformly in $0<h<h_0$ for a small enough $h_0$.
Since
\begin{align*}
&\int_{\dom}\intavg_{B_\eps(x+h\mathbf{e}_j)}\int_{U^{2}}|\xi_{1}-\xi_{2}|^q \,d\nu_{x,y}^2(\xi) \,dy \,dx \\
&\qquad\leq C \int_{\dom}\intavg_{B_\eps(x+h\mathbf{e}_j)}\int_{U^{2}}\left(|\xi_{1}|^q+|\xi_{2}|^q\right) \,d\nu_{x,y}^2(\xi) \,dy \,dx\\
&\stackrel{\text{consistency}}{=}2 C \int_{\dom}\int_{U}|\xi_1|^q \,d\nu_{x}^1(\xi_1) \,dx \leq C<\infty \quad \text{ for almost every } t,
\end{align*}
we can apply the dominated convergence theorem and pass the limit in $\eps$ inside:
\begin{multline}\label{eq:dctx}
\lim_{\eps\rightarrow 0}\frac{1}{h^q}\int_0^T\int_{\dom}\intavg_{B_\eps(x+h\mathbf{e}_j)}\int_{U^{2}}|\xi_{1}-\xi_{2}|^q \,d\nu_{x,y}^2(\xi) \,dy \,dx \,dt\\
 = \frac{1}{h^q}\int_0^T\lim_{\eps\rightarrow 0}\int_{\dom}\intavg_{B_\eps(x+h\mathbf{e}_j)}\int_{U^{2}}|\xi_{1}-\xi_{2}|^q \,d\nu_{x,y}^2(\xi) \,dy \,dx \,dt
\end{multline}

Now fix $\eps>0$, a `good' $t\in [0,T]$,
and note that for almost every $x\in \dom$,
\begin{align*}
&\left|\left(\int_{\dom}\intavg_{B_\eps(x+h\mathbf{e}_j)} \int_{U^{2}} |\xi_{1}-\xi_{2}|^q \,d\nu_{x,y}^2(\xi) \,dy \,dx\right)^{1/q} - \left(\int_{\dom} \int_{U^{2}} |\xi_{1}-\xi_{2}|^q \,d\nu_{x,x+h\mathbf{e}_j}^2(\xi) \,dx\right)^{1/q} \right|\\
&\qquad \stackrel{\text{consistency}}{=}\Biggl|\left(\int_{\dom}\intavg_{B_\eps(x+h\mathbf{e}_j)} \int_{U^{3}} |\xi_{1}-\xi_{2}|^q \,d\nu_{x,y,x+h\mathbf{e}_j}^3(\xi) \,dy \,dx\right)^{1/q}\\
&\qquad\hphantom{\stackrel{\text{consistency}}{=}\Bigl|}-\left(\int_{\dom}\intavg_{B_\eps(x+h\mathbf{e}_j)} \int_{U^{3}} |\xi_{1}-\xi_{3}|^q \,d\nu_{x,y,x+h\mathbf{e}_j}^3(\xi)\,dy \,dx\right)^{1/q}\Biggr|\\
&\qquad\stackrel{\Delta\text{-ineq.}}{\leq}\left(\int_{\dom}\intavg_{B_\eps(x+h\mathbf{e}_j)} \int_{U^{3}} |\xi_{3}-\xi_{2}|^q \,d\nu_{x,y,x+h\mathbf{e}_j}^3(\xi) \,dy \,dx\right)^{1/q}\\
&\qquad \stackrel{\text{consistency}}{=}\left(\int_{\dom}\intavg_{B_\eps(x+h\mathbf{e}_j)} \int_{U^{2}} |\xi_{2}-\xi_{1}|^q \,d\nu_{y,x+h\mathbf{e}_j}^2(\xi) \,dy \,dx\right)^{1/q}
\end{align*}
The last term goes to zero by the diagonal continuity property of $\cm$ (c.f.~\cite[Definition 2.5]{FLM17}). Hence, for almost every $t\in [0,T]$,
\begin{equation}\label{eq:weisnoed}
\lim_{\eps\to 0}\int_{\dom}\intavg_{B_\eps(x+h\mathbf{e}_j)} \int_{U^{2}} |\xi_{1}-\xi_{2}|^q \,d\nu_{x,y}^2(\xi) \,dy \,dx =\int_{\dom} \int_{U^{2}} |\xi_{1}-\xi_{2}|^q \,d\nu_{x,x+h\mathbf{e}_j}^2(\xi) \,dx.
\end{equation}
Thus, combining~\eqref{eq:weisnoed} and~\eqref{eq:dctx},
\begin{align*}
&\int_0^T \int_{L^q(\dom)}\int_{\dom}\frac{|u(x)-u(x+h\mathbf{e}_j)|^q}{h^q} \,dx \,d\mu_t(u) \,dt \\
&\qquad=\lim_{\eps\to 0}\frac{1}{h^q}\int_0^T\int_{L^q(\dom)}\int_{\dom}\intavg_{B_\eps(x+h\mathbf{e}_j)}|u(x)-u(y)|^q \,dy \,dx \,d\mu_t(u) \,dt\\
&\qquad =\lim_{\eps\rightarrow 0} \frac{1}{h^q} \int_0^T\int_{\dom}\intavg_{B_\eps(x+h\mathbf{e}_j)}\int_{U^{2}}|\xi_{1}-\xi_{2}|^q \,d\nu_{x,y}^2(\xi) \,dy \,dx \,dt\\
&\qquad = \frac{1}{h^q}\int_0^T\int_{\dom}\int_{U^{2}} |\xi_{1}-\xi_{2}|^q \,d\nu_{x,x+h\mathbf{e}_j}^2(\xi)\,dx \,dt.
\end{align*}
Since this last identity holds for any $h>0$, and the limit is bounded by assumption \eqref{eq:aprioribd}, we can pass to the limit $h\to 0$ and obtain, using Lemma \ref{lem:gradientrepresentation},
\begin{align*}
\int_0^T\int_{L^q(\dom)}\int_{\dom}|\partial_j u|^q \,dx \,d\mu_t(u) \,dt & = \lim_{h\rightarrow 0} \int_0^T \int_{L^q(\dom)}\int_{\dom}\frac{|u(x)-u(x+h\mathbf{e}_j)|^q}{h^q} \,dx \,d\mu_t(u) \,dt\\
& = \lim_{h\rightarrow 0}\frac{1}{h^q}\int_0^T\int_{\dom}\int_{U^{2}} |\xi_{1}-\xi_{2}|^q \,d\nu_{x,x+h\mathbf{e}_j}^2(\xi)\,dx \,dt.
\end{align*}
Summing over $j=1,\dots, d$, we obtain the claim.
\end{proof}
\begin{lemma}
	[Reverse direction of Lemma \ref{lem:gradientrep2}]\label{lem:gradientrep3}
	Let $\cm=(\cm_t)_{0\leq t\leq T}=((\nu^1_t,\nu_t^2\dots))_{0\leq t\leq T}$ be a correlation measure as in~\cite[Definition 2.5]{FLM17} that satisfies
	\begin{equation}\label{eq:aprioribdcm}
	\lim_{h\rightarrow 0}\frac{1}{h^q}\int_0^T \int_{\dom}\sum_{j=1}^d\int_{U^2}|\xi_{1}-\xi_2|^q \,d\nu^2_{x,x+h\mathbf{e}_j}(\xi) \,dx \,dt\leq  C<\infty.
	\end{equation}
	Let $\dom$ be the $d$-dimensional torus. Then $\cm$ corresponds to a family of probability measures $\mu=(\mu_t)_{0\leq t\leq T}$ on $L^q(\dom)$ that satisfy
	\begin{equation}\label{eq:corgradientagain}
	\int_0^T\int_{L^q(\dom)} \norm{\Grad u}_{L^q(\dom)}^q \,d\mu_t(u)\,dt = \lim_{h\rightarrow 0}\frac{1}{h^q}\int_0^T \int_{\dom}\sum_{j=1}^d\int_{U^2}|\xi_{1}-\xi_2|^q \,d\nu^2_{x,x+h\mathbf{e}_j}(\xi) \,dx \,dt\leq \infty.
	\end{equation}
\end{lemma}
\begin{proof}
	The proof follows from the proof of Lemma \ref{lem:gradientrep2} by going in the reverse direction and checking that all the steps are equivalences.
\end{proof}
\begin{lemma}\label{lem:divconstraintlem}
	Let $\Phi:H^1(\dom)\to \R$ be a function of the form
	\begin{equation*}
	\Phi(u)= \phi\left((u,\Grad g_1),\dots, (u,\Grad g_k)\right),
	\end{equation*}
	where $\phi:\R^k\to\R$ is bounded and continuous with $\phi(0)=0$ and $g_i\in H^1(\dom)$. Then
	\begin{equation*}
	\int_{L^2(\dom)}\Phi(u) \,d\mu_t(u) =0,
	\end{equation*}
	for any $\mu$ supported on a bounded set of $L^2(\dom)$ corresponding to a family of correlation measures $\cm=(\nu_1,\nu_2,\dots)$ satisfying \eqref{eq:divfreeconst}, i.e.,
		\begin{equation*}
	\int_{\dom^{k}}\int_{U^{k}} \xi_1\otimes\dots\otimes\xi_{\ell}\otimes\alpha_{\ell+1}(\xi_{\ell+1})\otimes\dots \otimes \alpha_k(\xi_k) \,d\nu_{t,x}^k(\xi)\cdot\Grad_{x_1,\dots, x_\ell}\varphi(x) \,dx = 0,
	\end{equation*}
	where $\Grad_{x_1,\dots,x_\ell}=(\Grad_{x_{1}},\dots,\Grad_{x_\ell})^\top$, $1\leq \ell\leq k\in\N$, for all $\varphi\in H^1(\dom^k;U^{k-\ell})$, $\alpha_j\in C(U;U)$, with $\alpha_j(v)\leq C(1+|v|^2)$ for all $j=1,\dots, k$.
\end{lemma}
\begin{proof} 
	Since $\phi$ is bounded and continuous, we can approximate it on every compact subset $B\subset \R^k$ by polynomials $p_n$ such that
	\begin{equation*}
	\norm{\phi-p_n}_{C(B)}\leq \frac{1}{n},
	\end{equation*}
	$n\in \N$. Since $\phi(0)=0$, we may assume that $p_n(0)=0$ and hence the constant term of the polynomial is zero. It is therefore of the form
	\begin{equation*}
	p_n(\zeta) = \sum_{|\underline{j}|=1}^{N_n} \beta_{\underline{j}} \zeta_1^{j_1}\dots \zeta_k^{j_k},
	\end{equation*}
	where $N_n\in\N$ large enough, $\zeta=(\zeta_1,\dots,\zeta_k)$, $\underline{j}=(j_1,\dots,j_k)$, $j_i\geq 0$, and $|\underline{j}| =j_1+\dots+j_k$.
	Hence
	\begin{equation*}
	\int_{L^2(\dom)}p_n((u,\Grad g_1),\dots,(u,\Grad g_k)) \,d\mu_t(u)
	= \sum_{|\underline{j}|=1}^{N_n} \beta_{\underline{j}}	\int_{L^2(\dom)}(u,\Grad g_1)^{j_1}\cdots(u,\Grad g_k)^{j_k} \,d\mu_t(u).
	\end{equation*}
	Consider one of the terms in the sum:
	\begin{align*}
	&	\int_{L^2(\dom)}(u,\Grad g_1)^{j_1}\cdots(u,\Grad g_k)^{j_k} \,d\mu_t(u) \\
	& = \int_{L^2(\dom)}\int_D u(x_{1})\cdot\Grad_{x_1} g_1(x_1) \,dx_1\cdots \int_D u(x_{j_1})\cdot\Grad_{x_{j_1}} g_1(x_{j_1}) \,dx_{j_1} \\
	&\qquad\cdot\int_D u(x_{j_1+1})\cdot\Grad_{x_{j_1+1}} g_2(x_{j_1+1}) \,dx_{j_1+1} \cdots  \int_D u(x_{|\underline{j}|})\cdot\Grad_{x_{|\underline{j}|}} g_k(x_{|\underline{j}|}) d x_{|\underline{j}|}  \,d\mu_t(u)\\
	& = \int_{D^{|\underline{j}|}}\int_{(\R^{N})^{|\underline{j}|}}\xi_1\otimes\dots\otimes \xi_{|\underline{j}|} \,d\nu^{|\underline{j}|}_x(\xi) \Grad_{x_{1}}g_1(x_1)\dots \Grad_{x_{|\underline{j}|}} g_k(x_{|\underline{j}|}) \,dx =0,
	\end{align*}
	where the last equality follows from \eqref{eq:divfreeconst}. Hence
	\begin{equation*}
	\int_{L^2(\dom)}p_n((u,\Grad g_1),\dots,(u,\Grad g_k)) \,d\mu_t(u)
	= 0.
	\end{equation*}
	By the approximation property of the $p_n$, after passing $n\to\infty$, we obtain the result for arbitrary continuous $\phi$.
\end{proof}
\section{An identity by T.\ Drivas for second order structure functions}\label{sec:appendixC}
We present the proof of the following lemma discovered by Theo Drivas~\cite[Lemma 1]{Drivas2021} (Lemma~\ref{lem:Drivas}):
\begin{lemma}[Weak anisotropy]
	Let $\mu_t$ be a statistical solution of the Navier--Stokes equation. Then $\mu$ satisfies
	\begin{equation}\label{eq:WA}
	\begin{split}
	&3\int_0^T\int_{\dom}\int_{\hdiv}\fint_{\partial B_r(0)} (\delta_{rn} u\cdot n)^2dS(n)\,dx \,d\mu_t(u)\,dt\\
	&=\int_0^T\int_{\dom}\int_{\hdiv}\fint_{B_r(0)} |\delta_\ell u(x)|^2 d\ell \,dx \,d\mu_t(u)\,dt .
	\end{split}
	\end{equation}
\end{lemma}

\begin{proof}
We let $\omega\in C^\infty_c(\R^d)$, $\omega\geq 0$, compactly supported in $B_1(0)$ and $\omega_\eta(y)\coloneqq\eta^{-d}\omega(y/\eta)$ for $\eta>0$ its rescaled version. For some fixed $\ell\in \R^3$ and $x\in D$, we let $\phi(t,y) = \ell \omega_\eta(x-y)$. We use this as a test function in~\eqref{eq:FKsyspressure} for $k=1$. Using the Einstein summation convention we then have
\begin{equation}
\begin{split}
&\int_0^T\int_{\hdiv}\int_{\dom} u(y)\cdot \ell \omega_\eta(x-y)\theta'(t) \,dy \,d\mu_t(u)\,dt\\
&-\int_0^T\int_{\hdiv}\int_{\dom} u^i(y)u^j(y)\ell^j \partial_{y^i}\omega_\eta(x-y)\theta(t)\,dy \,d\mu_t(u)\,dt \\
&+ \vis \int_0^T\int_{\hdiv} \int_{\dom} u^i(y)\ell^i\Delta_y\omega_\eta(x-y)\theta(t) \,dy\,d\mu_t(u) \,dt \\
&+ \int_0^T\int_{\hdiv}\int_{\dom} u^i(y)u^j(y) \partial_{y^i}\partial_{y^j}\psi_{\omega}(x-y)\theta(t)\,dy \,d\mu_t(u)\,dt=0.
\end{split}
\end{equation}
where the last term comes from the fact that the test function is not divergence free, c.f.\ Remark~\ref{rem:lerayprojection}. We define $u_\eta(x)\coloneqq\int_{\dom} u(y)\omega_{\eta}(x-y) \,dy$ and rewrite the above in terms of $u_\eta$,
\begin{equation}
\begin{split}
&\int_0^T\int_{\hdiv} u_\eta(x)\cdot \ell \theta'(t) \,d\mu_t(u)\,dt+\int_0^T\int_{\hdiv}u^i_\eta(x)\partial_{x^i} u^j_\eta(x)\ell^j \,d\mu_t(u)\theta(t)\,dt \\
&+ \vis \int_0^T\int_{\hdiv} \Delta_x u^i_\eta(x)\ell^i\theta(t) \,d\mu_t(u) \,dt \\
= &-\int_0^T\int_{\hdiv}\int_{\dom} u^i(y)u^j(y) \partial_{y^i}\partial_{y^j}\psi_{\omega}(x-y)\theta(t)\,dy \,d\mu_t(u)\,dt\\
&+\int_0^T\int_{\hdiv}[u^i_\eta(x)\partial_{x^i} u^j_\eta(x)- \partial_{x^i}((u^iu^j)_{\eta}(x))]\ell^j \,d\mu_t(u)\theta(t)\,dt .
\end{split}
\end{equation}
We subtract this identity from the same identity at $x+\ell$ (and use the notation $\delta_\ell u=u(x+\ell)-u(x)$):
\begin{equation}\label{eq:lalala}
\begin{split}
&\int_0^T\int_{\hdiv} \delta_\ell u_\eta(x)\cdot \ell \theta'(t) \,d\mu_t(u)\,dt\\
&+\int_0^T\int_{\hdiv}[u^i_\eta(x+\ell)\partial_{x^i} u^j_\eta(x+\ell)-u^i_\eta(x)\partial_{x^i} u^j_\eta(x)]\ell^j \,d\mu_t(u)\theta(t)\,dt \\
&+ \vis \int_0^T\int_{\hdiv} \Delta_x\delta_\ell u^i_\eta(x)\ell^i\theta(t) \,d\mu_t(u) \,dt \\
= &-\int_0^T\int_{\hdiv}\int_{\dom} u^i(y)u^j(y) \partial_{y^i}\partial_{y^j}\psi_{\omega}(x+\ell-y)\theta(t)\,dy \,d\mu_t(u)\,dt\\
&+\int_0^T\int_{\hdiv}\int_{\dom} u^i(y)u^j(y) \partial_{y^i}\partial_{y^j}\psi_{\omega}(x-y)\theta(t)\,dy \,d\mu_t(u)\,dt\\
&+\int_0^T\int_{\hdiv}[u^i_\eta(x+\ell)\partial_{x^i} u^j_\eta(x+\ell)- \partial_{x^i}((u^iu^j)_{\eta}(x+\ell))]\ell^j \,d\mu_t(u)\theta(t)\,dt \\
&-\int_0^T\int_{\hdiv}[u^i_\eta(x)\partial_{x^i} u^j_\eta(x)- \partial_{x^i}((u^iu^j)_{\eta}(x))]\ell^j \,d\mu_t(u)\theta(t)\,dt .
\end{split}
\end{equation}
Denote the above terms as $E_1,E_2,\dots,E_7$, which are all functions of $x$. We now integrate~\eqref{eq:lalala} over $x\in\dom$ and consider each term separately. As we will see, all of the terms in \eqref{eq:lalala} vanish. We obtain
\begin{equation}
\int_D E_1\,dx = \int_0^T\int_{\hdiv}\int_{\dom} \delta_\ell u_\eta(x)\cdot \ell \,dx\theta'(t) \,d\mu_t(u)\,dt = 0,
\end{equation}
after changing the integration from $x$ to $x+\ell$ in one of the integrals. Next,
\begin{equation}
\int_D E_3\,dx = \vis \int_0^T\int_{\hdiv}\int_{\dom} \Delta_x u^i_\eta(x)\ell^i \,dx\theta(t) \,d\mu_t(u) \,dt =0,
\end{equation}
again because of periodic boundary conditions. We also get
\begin{equation}
\begin{split}
\int_\dom E_4\,dx &=\int_0^T\int_{\hdiv}\iint_{\dom^2} u^i(y)u^j(y) \partial_{y^i}\partial_{y^j}\psi_{\omega}(x+\ell-y)\theta(t)\,dy\,dx  \,d\mu_t(u)\,dt\\
  &=\int_0^T\int_{\hdiv}\iint_{\dom^2} u^i(y)u^j(y) \partial_{x^i}\partial_{x^j}\psi_{\omega}(x+\ell-y)\theta(t)\,dy\,dx  \,d\mu_t(u)\,dt\\
 &=\int_0^T\int_{\hdiv}\int_{\dom} u^i(y)u^j(y) \int_{\dom}\partial_{x^i}\partial_{x^j}\psi_{\omega}(x+\ell-y)\,dx\theta(t)\,dy  \,d\mu_t(u)\,dt=0,
\end{split}
\end{equation}
and similarly, $\int_\dom E_5\,dx = 0$. We also have
\begin{equation}
\begin{split}
\int_\dom E_6\,dx &=\int_0^T\int_{\hdiv}\int_{\dom}[u^i_\eta(x+\ell)\partial_{x^i} u^j_\eta(x+\ell)- \partial_{x^i}(u^iu^j)_{\eta}(x+\ell)]\ell^j\,dx  \,d\mu_t(u)\theta(t)\,dt \\
&=\int_0^T\int_{\hdiv}\int_{\dom}\partial_{x^i}\big[\big(u^i_\eta(x+\ell) u^j_\eta(x+\ell)\big)- (u^iu^j)_{\eta}(x+\ell)\big]\ell^j\,dx  \,d\mu_t(u)\theta(t)\,dt \\
&= 0
\end{split}
\end{equation}
from the fact that $u$ is divergence free and the identity $\int_D \partial_{x^i} v\,dx = 0$ for every $v\in C^1(D)$, which follows from the periodic boundary conditions. In the same manner we obtain $\int_\dom E_7\,dx = 0$. Since $E_2$ is the only remaining term in \eqref{eq:lalala}, we obtain also
\begin{equation}\label{eq:lalalala}
\int_\dom E_2\,dx = 0.
\end{equation}

We next rewrite $E_2$ as
\begin{equation}
\begin{split}
E_2&=\int_0^T\int_{\hdiv}\big[u^i_\eta(x+\ell)\partial_{x^i} u^j_\eta(x+\ell)-u^i_\eta(x)\partial_{x^i} u^j_\eta(x)\big]\ell^j \,d\mu_t(u)\theta(t)\,dt \\
& =\int_0^T\int_{\hdiv}\big[u^i_\eta(x)\partial_{x^i} \delta_\ell u^j_\eta(x) + \delta_\ell u^i_\eta(x)\partial_{x^i} u^j_\eta(x+\ell)\big]\ell^j \,d\mu_t(u)\theta(t)\,dt \\
& =\int_0^T\int_{\hdiv}\big[u^i_\eta(x)\partial_{x^i} \delta_\ell u^j_\eta(x) + \delta_\ell u^i_\eta(x)\partial_{\ell^i} u^j_\eta(x+\ell)\big]\ell^j \,d\mu_t(u)\theta(t)\,dt \\
& =\int_0^T\int_{\hdiv}\big[u^i_\eta(x)\partial_{x^i} \delta_\ell u^j_\eta(x) + \delta_\ell u^i_\eta(x)\partial_{\ell^i} \delta_\ell u^j_\eta(x)\big]\ell^j \,d\mu_t(u)\theta(t)\,dt \\
& =\int_0^T\int_{\hdiv}\big[u^i_\eta(x)\partial_{x^i} \delta_\ell u^j_\eta(x)\ell^j + \delta_\ell u^i_\eta(x)\partial_{\ell^i}\big(\delta_\ell u^j_\eta(x)\ell^j\big)-\delta_\ell u^i_\eta(x) \delta_\ell u^j_\eta(x)\partial_{\ell^i}\ell^j\big] \,d\mu_t(u)\theta(t)\,dt \\
& =\int_0^T\int_{\hdiv}\big[u^i_\eta(x)\partial_{x^i} \delta_\ell u^j_\eta(x)\ell^j + \delta_\ell u^i_\eta(x)\partial_{\ell^i}\big(\delta_\ell u^j_\eta(x)\ell^j\big)-\big|\delta_\ell u_\eta(x)\big|^2\big] \,d\mu_t(u)\theta(t)\,dt\\
&= E_{2,1}+E_{2,2}+E_{2,3}.
\end{split}
\end{equation}
Integrating these terms, we have
\begin{equation}
\begin{split}
\int_\dom E_{2,1}\,dx &=\int_0^T\int_{\hdiv}\int_{\dom}u^i_\eta(x)\partial_{x^i} \delta_\ell u^j_\eta(x)\ell^j \,dx\,d\mu_t(u)\theta(t) \,dt\\
& = \int_0^T\int_{\hdiv}\int_{\dom}\partial_{x^i}\big(u^i_\eta(x) \delta_\ell u^j_\eta(x)\ell^j\big) \,dx\,d\mu_t(u)\theta(t) \,dt=0,
\end{split}\end{equation}
where the first equation follows because $\mu$ is supported on divergence free functions. Again using the divergence free property we can write
\begin{equation}
\int_\dom E_{2,2} \,dx = \int_0^T\int_{\dom}\int_{\hdiv} \partial_{\ell^i}\big(\delta_\ell u^i_\eta(x) \delta_\ell u^j_\eta(x)\ell^j\big) \,d\mu_t(u)\theta(t)\,dx \,dt.
\end{equation}
Now we integrate in $\ell$ over a ball of radius $r$, take $\theta(t)\equiv 1$, and use Gauss' divergence theorem:
\begin{equation}
\begin{split}
\fint_{B_r(0)} E_{2,2}\,d\ell &=\int_0^T\int_{\dom}\int_{\hdiv}\fint_{B_r(0)} \partial_{\ell^i}\big(\delta_\ell u^i_\eta(x) \delta_\ell u^j_\eta(x)\ell^j\big) d\ell \,d\mu_t(u)\,dx\theta(t)\,dt\\
&=3\int_0^T\int_{\hdiv}\int_{\dom}\fint_{\partial B_r(0)} \delta_{rn} u^i_\eta(x)n^i \delta_{rn} u^j_\eta(x)n^j dS(n)\,dx \,d\mu_t(u)\,dt\\
&=3\int_0^T\int_{\hdiv}\int_{\dom}\fint_{\partial B_r(0)} \big(\delta_{rn} u_\eta\cdot n\big)^2 dS(n)\,dx \,d\mu_t(u)\,dt,
\end{split}
\end{equation}
the factor 3 coming from the relation $|B_r(0)| = 3^{-1}r|\partial B_r(0)|$. Since $\mu$ has bounded support in $L^2$, we can let $\eta\to 0$ in \eqref{eq:lalalala} and obtain
\begin{equation}
\begin{split}
&3\int_0^T\int_{\dom}\int_{\hdiv}\fint_{\partial B_r(0)} (\delta_{rn} u\cdot n)^2dS(n)\,dx \,d\mu_t(u)\,dt\\
&=\int_0^T\int_{\dom}\int_{\hdiv}\fint_{B_r(0)} |\delta_\ell u(x)|^2 d\ell \,dx \,d\mu_t(u)\,dt,
\end{split}
\end{equation}
which is~\eqref{eq:WA}.
\end{proof}
\bibliographystyle{abbrv}
\bibliography{GCMbib}

\end{document}